\documentclass[final,hidelinks]{siamart220329}
\usepackage[T1]{fontenc}
\usepackage{amssymb,bm}
\usepackage{booktabs}
\usepackage{arydshln}
\usepackage{placeins}
\usepackage{float}
\graphicspath{{./Figures/}}

\makeatletter
\def\compactdoi#1{\expandafter\compactdoiaux#1\@nil{#1}}
\def\compactdoiaux https://doi.org/#1\@nil#2{\href{#2}{\nolinkurl{doi:#1}}}
\renewcommand{\url}[1]{\compactdoi{#1}}
\makeatother

\newcommand{\ii}{\mathrm{i}}
\newcommand{\Eh}{\bm{E}_h}
\newcommand{\curl}{\nabla\times}
\newcommand{\Vh}{V_h}

\headers{A Three-Grid Preconditioner for Maxwell Equations}{Shubin Fu}
\begin{document}
	
	\title{A Fully Matrix-Free Three-Grid Preconditioner for the Time-Harmonic Maxwell Equations at Extreme Scale}
	\author{Shubin Fu\thanks{School of Mathematical Sciences, Eastern Institute of
		Technology, Ningbo, Zhejiang 315200, P. R. China
		(\email{sfu@eitech.edu.cn}).}}
	\maketitle
	
				\begin{abstract}
Three-dimensional time-harmonic Maxwell simulations generate massive complex indefinite systems whose mesh coarsening is strictly limited by phase accuracy. Although matrix-free finite element kernels utilize GPU throughput efficiently, standard multilevel solvers are ultimately bottlenecked by the memory and communication costs of exact coarse-grid factorizations. We present a fully matrix-free, factorization-free three-grid preconditioner for curl-conforming  N{\'e}delec discretizations with perfectly matched layers (PML) and optimally blended quadrature. The method employs an outer FGMRES to solve the unshifted fine-grid equation, while an intermediate-grid correction is computed by a fixed-work FGMRES preconditioned with a complex-shifted $2h$--$4h$ cycle. This strategically confines the complex shift to an auxiliary preconditioner, preserving the physical Maxwell operator. A local Fourier analysis derives the blended Maxwell branches and compatible edge transfers, identifying robust shift and Jacobi damping parameters. Validated against the analytical Maxwell Green tensor, our approach demonstrates extreme scalability: using a single solver configuration, both homogeneous and heterogeneous systems with approximately 10.89 billion complex edge unknowns are solved in 42.0--72.0 seconds on just 64 NVIDIA A100 GPUs.
				\end{abstract}
	
	
	
	\section{Introduction}
	
Time-harmonic Maxwell equations arise in radar scattering, antenna design, photonics, metamaterials, and electromagnetic imaging \cite{Monk2003,Tsuji2012,Chanaud2014}. When a three-dimensional domain spans many wavelengths, maintaining a fixed spatial resolution forces the number of unknowns to grow cubically with the propagation distance. The resulting linear systems are inherently challenging: the curl--curl operator possesses a massive gradient kernel, the wave term renders the transverse branch indefinite, and perfectly matched layer (PML) truncation introduces complex anisotropic coefficients \cite{ReviewBerenger1994,ReviewChewJinMichielssen1997,GopalakrishnanPasciakDemkowicz2004}.

To control accumulated phase error across long propagation paths without artificially inflating the system size, high-frequency Maxwell simulations must balance dispersion accuracy with linear solver efficiency. In this work, we adopt curl-conforming  N{\'e}delec finite elements, which strictly preserve tangential continuity and the underlying de Rham complex \cite{Nedelec1980,Monk2003}. Furthermore, we employ optimally blended quadrature to suppress dispersion without introducing auxiliary unknowns or altering the finite element space \cite{AinsworthWajid2010,Wajid2012}. Although this discretization yields superior phase fidelity, it creates a severe numerical bottleneck: how to efficiently invert the resulting enormous, complex indefinite system at extreme scales.

Existing linear solvers struggle to reconcile algorithmic scalability with high parallel throughput on modern computing architectures. Sparse direct solvers provide exceptional robustness, but their memory fill-in, arithmetic complexity, and communication overhead grow superlinearly in three dimensions. While hierarchical compression techniques mitigate these demands \cite{ReviewBanjaiHackbusch2008,ReviewEngquistYing2011HMat}, they do not eliminate the global factorization bottleneck. Domain-decomposition and sweeping preconditioners \cite{IntroEngquistYing2011PML,ReviewStolk2013,ReviewVionGeuzaine2014,ReviewZepedaNunezDemanet2016,Tsuji2012} achieve low iteration counts by sweeping wave updates across subdomains, and their parallel extensions expose notable concurrency \cite{IntroPoulson2013,ReviewLiuYing2016Additive,ReviewLiuYing2016Recursive,ReviewStolk2017}. Nevertheless, their intrinsic directional dependencies, factorization storage, and subproblem synchronizations remain poorly aligned with the massive matrix-free parallel throughput of GPUs.

Multilevel geometric methods provide global coarse-grid communication without directional sweeping, making them naturally suited for GPU execution. Parallel geometric multigrid has successfully reached 1.3 billion unknowns for time-harmonic Maxwell systems \cite{Chanaud2014}. However, standard multigrid frameworks still rely on assembling and factorizing the coarsest operator, creating an unavoidable bottleneck that eventually stalls scaling. Although recent advances have further enriched compatible block, anisotropic, and cascadic components \cite{SyedFarquharsonMacLachlan2020,SunZhangChenFengHu2022,KinnewigRothWick2021,WangPanWu2025}, coarse-grid factorization remains the primary obstacle to achieving extreme-scale computations.

While complex-shifted operators can stabilize multigrid cycles for indefinite wave equations \cite{ReviewErlanggaVuikOosterlee2004,ReviewErlanggaOosterleeVuik2006,ReviewCoolsVanroose2013}, a shift fundamentally alters the physical spectrum. Resolving this tension for the full 3D vector Maxwell system---where the $H(\mathrm{curl})$ gradient kernel and PML anisotropy compound solver instability---demands a novel architectural design. In this paper, we establish a fully matrix-free and factorization-free three-grid preconditioner specifically tailored for curl-conforming Maxwell discretizations. Instead of substituting the physical Maxwell operator or relying on exact coarse factorizations \cite{Pinel2010,Calandra2013}, our solver strategically decouples the physical problem from an auxiliary shifted hierarchy.

Specifically, an outer FGMRES targets the unshifted fine-grid system, while its intermediate-grid correction is computed by an inner FGMRES preconditioned with a complex-shifted $2h$--$4h$ cycle. The complex shift is thus strictly confined to an auxiliary preconditioner, leaving the primary physical operator untouched. Supported by a comprehensive local Fourier analysis (LFA), we derive the blended Maxwell branches and compatible edge transfers, prove harmonic closure, and identify robust parameter ranges for the shift and Jacobi damping. Because every level operates strictly through geometric edge transfers and fixed-work Krylov smoothers, the memory footprint remains strictly proportional to the number of edge unknowns, eliminating factorization storage and inner convergence checks entirely.

The resulting algorithm demonstrates unprecedented performance and scalability. Managed by a single, uniform solver configuration, our implementation tackles homogeneous, converging-lens, periodic, and random isotropic media on up to 64 NVIDIA A100-40G GPUs. The solver successfully handles systems containing approximately 10.89 billion complex edge unknowns, completing the solves in just 42.0 to 72.0 seconds with matrix-free setup times under 0.8 seconds and peak memory below 37.7 GiB per GPU. To the best of our knowledge, these calculations represent the largest reported 3D time-harmonic Maxwell simulations utilizing curl-conforming finite elements \cite{Chanaud2014}.

The remainder of this paper is organized as follows. Section~2 formulates the PML-truncated Maxwell problem and its optimally blended  N{\'e}delec discretization. Section~3 presents the algorithmic structure of the matrix-free three-grid preconditioner. Section~4 details the local Fourier analysis and parameter characterization. Sections~5 and~6 present numerical validation against analytical Green's tensors and large-scale GPU performance benchmark results, respectively.

	\section{Maxwell Problem and Discretization}
	\label{sec:maxwell-discretization}

In this section, we formulate the time-harmonic Maxwell problem with Cartesian PML truncation and establish its discrete counterpart using curl-conforming  N{\'e}delec finite elements with optimally blended quadrature. 
	\subsection{Radiation problem}

		Following \cite{Monk2003}, let $\omega>0$ and adopt the time convention
		$\exp(-\ii\omega t)$.  The electric field $\bm{E}$ and magnetic field $\bm{H}$
		generated by an impressed current $\bm{J}$ satisfy
	\[
	\begin{aligned}
		\curl\bm{E} &= \ii\omega\mu_0\mu_r\bm{H},\\
		\curl\bm{H} &= -\ii\omega\epsilon_0\epsilon_r\bm{E}+\bm{J}
	\end{aligned}
	\qquad\text{in }\mathbb{R}^3.
	\]
	Here $\epsilon_0$ and $\mu_0$ are the free-space constants, and
	$\epsilon_r(\bm{x})$ and $\mu_r(\bm{x})$ are the relative permittivity and
	permeability.  The media considered below are isotropic in the physical domain,
	so these coefficients are scalar functions before the PML transformation.
	Eliminating $\bm{H}$ gives the electric-field equation
	\begin{equation}
		\curl\bigl(\mu_r^{-1}\curl\bm{E}\bigr)
		-k_0^2\epsilon_r\bm{E}=\bm{f},
		\qquad
		k_0=\omega\sqrt{\epsilon_0\mu_0},
		\qquad
		\bm{f}=\ii\omega\mu_0\bm{J}.
		\label{eq:maxwell-electric-exterior}
	\end{equation}
	In a homogeneous exterior, the outgoing solution is selected by the
	Silver--Muller condition
	\begin{equation}
		\lim_{r\to\infty}
		r\bigl(\bm{H}\times\widehat{\bm r}-\eta_0^{-1}\bm{E}\bigr)=0,
		\qquad
		r=|\bm{x}|,
		\quad
		\widehat{\bm r}=\bm{x}/r,
		\quad
		\eta_0=\sqrt{\mu_0/\epsilon_0}.
		\label{eq:silver-muller}
	\end{equation}
	To obtain a finite computational problem while retaining the outgoing solution
	in $\Omega_{\rm ph}$, we replace the condition at infinity by a Cartesian PML\@.

	\subsection{Cartesian PML truncation}

	Let $\Omega_{\rm ph}$ denote the physical domain and let the rectangular box
	$\Omega\supset\Omega_{\rm ph}$ include a Cartesian PML\@.  For the Cartesian
	coordinate $x_j$, $j\in\{1,2,3\}$, introduce the complex stretch
	\begin{equation}
		\widetilde x_j(x_j)=\int_0^{x_j}\widetilde s_j(t)\,dt,
		\qquad
		\widetilde s_j(x_j)=1+\ii\zeta_j(x_j),
		\label{eq:pml-stretch}
	\end{equation}
	where $\zeta_j=0$ in $\Omega_{\rm ph}$ and $\zeta_j>0$ in the
	corresponding PML\@.  We use the quadratic profile
	\begin{equation}
		\zeta_j(x_j)
		=
		\zeta_{j,\max}
		\left(\frac{d_j(x_j)}{d_{\rm pml}}\right)^2,
		\label{eq:pml-profile}
	\end{equation}
	where $d_j(x_j)$ is the distance into the layer and $d_{\rm pml}$ is
	its thickness.  With the adopted time convention, this choice attenuates an
	outgoing wave as it crosses the layer.

	Writing $(\widetilde s_x,\widetilde s_y,\widetilde s_z)
	=(\widetilde s_1,\widetilde s_2,\widetilde s_3)$, the coordinate transformation can be
	represented on the real domain by
	complex anisotropic material tensors
	\cite{ReviewBerenger1994,ReviewChewJinMichielssen1997,Tsuji2012}.  For an
	isotropic physical medium, set
	\begin{equation}
		S=\operatorname{diag}\left(
		\frac{\widetilde s_y\widetilde s_z}{\widetilde s_x},
		\frac{\widetilde s_x\widetilde s_z}{\widetilde s_y},
		\frac{\widetilde s_x\widetilde s_y}{\widetilde s_z}
		\right),
		\qquad
		\widetilde\epsilon_r=\epsilon_r S,
		\qquad
		\widetilde\mu_r=\mu_r S.
		\label{eq:pml-tensors}
	\end{equation}
	The truncated electric-field problem is therefore
	\begin{equation}
	\begin{aligned}
		\curl\bigl(\widetilde\mu_r^{-1}\curl\bm{E}\bigr)
		-k_0^2\widetilde\epsilon_r\bm{E}
		&=\bm{f} &&\text{in }\Omega,\\
		\bm{n}\times\bm{E}&=0 &&\text{on }\partial\Omega.
	\end{aligned}
		\label{eq:maxwell-pml-strong}
	\end{equation}
		Here $\bm n$ is the outward unit normal to $\partial\Omega$.  The outer
		perfect-conductor condition closes the finite PML box.  Since
	$S=I_3$ in $\Omega_{\rm ph}$, no boundary condition is introduced at the
	physical--PML interface.  Equation~\eqref{eq:maxwell-pml-strong} therefore
	provides a finite-domain problem with the same curl structure as the original
	Maxwell equation, and it leads directly to the weak formulation used below.

	\subsection{Variational formulation}

	We now formulate the PML-truncated problem in its natural energy space.  Define
	\[
		H_0(\mathrm{curl};\Omega)
		=
		\{\bm{v}\in L^2(\Omega)^3:
		\curl\bm{v}\in L^2(\Omega)^3,\ 
		\bm{n}\times\bm{v}=0\text{ on }\partial\Omega\}.
	\]
	With $(\bm{u},\bm{v})_\Omega
	=\int_\Omega\bm{u}\cdot\overline{\bm{v}}\,d\bm{x}$, the weak problem is to
	find $\bm{E}\in H_0(\mathrm{curl};\Omega)$ such that
	\begin{equation}
		a(\bm{E},\bm{v})
		:=
		(\widetilde\mu_r^{-1}\curl\bm{E},\curl\bm{v})_\Omega
		-k_0^2(\widetilde\epsilon_r\bm{E},\bm{v})_\Omega
		=(\bm{f},\bm{v})_\Omega
		\label{eq:maxwell-pml-weak}
	\end{equation}
	for every $\bm{v}\in H_0(\mathrm{curl};\Omega)$.  The negative mass term makes this
	problem indefinite at high frequency, while the PML tensors make it complex
	and non-Hermitian.  Its finite element approximation must therefore conform to
	$H_0(\mathrm{curl};\Omega)$ and retain the complex coefficients in
	\eqref{eq:maxwell-pml-weak}.  We now specify the discrete space and the
	quadrature used to evaluate this form.

	\subsection{Lowest-order N{\'e}delec discretization}

	Let $\mathcal{T}_h$ be a Cartesian hexahedral mesh of $\Omega$.  We discretize
	the electric field with the lowest-order first-family N{\'e}delec element
	\cite{Nedelec1980,Monk2003}, whose interelement continuity agrees with the
	tangential trace of $H(\mathrm{curl})$.
	
	On the reference cube $\widehat K=(-1,1)^3$, the local space is
	\begin{equation}
		\widehat V
		=Q_{0,1,1}\bm{e}_x
		\oplus Q_{1,0,1}\bm{e}_y
		\oplus Q_{1,1,0}\bm{e}_z,
		\label{eq:nedelec-local-space}
	\end{equation}
	where $Q_{a,b,c}$ contains polynomials of coordinatewise degrees at most
	$a$, $b$, and $c$.  Its twelve canonical degrees of freedom are the oriented
	edge moments
	\begin{equation}
			\ell_e(\bm{v})=\int_e\bm{v}\cdot\bm{t}_e\,ds.
			\label{eq:nedelec-edge-dof}
		\end{equation}
		Here $\bm t_e$ is the unit tangent consistent with the chosen orientation of
		edge $e$.
	For an element map $F_K:\widehat K\to K$ with Jacobian $J_K$, the covariant
	transformation
	\begin{equation}
		\bm{\psi}(\bm{x})
		=J_K^{-T}\widehat{\bm{\psi}}(\widehat{\bm{x}}),
		\qquad
		\bm{x}=F_K(\widehat{\bm{x}}),
		\label{eq:nedelec-covariant-map}
	\end{equation}
	preserves these moments and hence tangential continuity across element faces.
	A consistent global edge orientation gives the conforming space
	$\Vh\subset H_0(\mathrm{curl};\Omega)$.
	
	Let $\{\bm{\phi}_i\}$ be the global edge basis used for the coefficient
	representation; its levelwise normalization is specified in
	Section~\ref{sec:grid-hierarchy}.  Write
	$\Eh=\sum_i U_i\bm{\phi}_i$.  If $a_h$ denotes the quadrature approximation of
	\eqref{eq:maxwell-pml-weak}, then the discrete system is
	\begin{equation}
		A_h\bm{U}_h=\bm{F}_h,
		\qquad
		(A_h)_{ij}=a_h(\bm{\phi}_j,\bm{\phi}_i),
		\qquad
		(\bm{F}_h)_i=(\bm{f},\bm{\phi}_i)_\Omega.
		\label{eq:discrete-maxwell-system}
	\end{equation}
	The negative mass term and the complex PML tensors make $A_h$ indefinite and
	complex symmetric.
	
	\subsection{Optimally blended quadrature}
	\label{sec:blended-quadrature}

	Having fixed the $H(\mathrm{curl})$ space, the remaining discretization choice is the
	evaluation of the two terms in $a(\cdot,\cdot)$.  We use the lowest-order
	optimally blended finite-spectral quadrature derived for the time-harmonic
	Maxwell equations on tensor-product grids
	\cite{AinsworthWajid2010,Wajid2012}.  For polynomial order $p=1$, the optimal
	blending parameter is $p/(p+1)=1/2$.  The blend is implemented directly,
	without assembling and averaging two element matrices, by the nonstandard
	two-point rule
	\begin{equation}
		\widehat x_\pm=\pm\sqrt{\frac{2}{3}},
		\qquad
		w_+=w_-=1
		\label{eq:blended-rule-1d}
	\end{equation}
	on $[-1,1]$.  Element integrals on $\widehat K$ use the tensor product of
	\eqref{eq:blended-rule-1d}; neither the finite element space nor its degrees of
	freedom are changed.
	
	The dispersion expansion explains this choice \cite{Wajid2012}.  Consider a
	uniform Cartesian grid in three dimensions.  Let $\bm{\xi}$ be the wave vector,
	$h_j$ the spacing in direction $j$, and $h=\max_j h_j$.  Denote the
	transverse branch under Gauss and blended quadrature by
	$\lambda_{h,{\rm G}}(\bm\xi)$ and
	$\lambda_{h,{\rm blend}}(\bm\xi)$, respectively.
	Standard Gauss integration gives
	\begin{equation}
		\lambda_{h,{\rm G}}(\bm{\xi})
		=|\bm{\xi}|^2
		+\frac{1}{12}\sum_{j=1}^3 h_j^2\xi_j^4
		+O(h^4|\bm{\xi}|^6),
		\label{eq:dispersion-gauss}
	\end{equation}
	whereas the optimally blended rule cancels the quadratic term and yields
	\begin{equation}
		\lambda_{h,{\rm blend}}(\bm{\xi})
		=|\bm{\xi}|^2
		-\frac{1}{240}\sum_{j=1}^3 h_j^4\xi_j^6
		+O(h^6|\bm{\xi}|^8).
		\label{eq:dispersion-blended}
	\end{equation}
		Thus the leading dispersion error is fourth rather than second order, without
		increasing the number of unknowns.  We use the same element rule in the
		heterogeneous physical region and for the complex PML coefficients.  Repeating
	this construction on the $h$, $2h$, and $4h$ meshes produces the three physical
	Maxwell operators used by the preconditioner in the next section.
	
	\section{Three-Grid Preconditioner}
	\label{sec:three-grid}

In this section, we present the fully matrix-free three-grid preconditioner for \eqref{eq:discrete-maxwell-system}. The preconditioner couples an unshifted physical hierarchy on the $h$ and $2h$ meshes with a shifted auxiliary cycle on the $2h$ and $4h$ meshes.

	\subsection{Grid hierarchy and operators}
	\label{sec:grid-hierarchy}

	Let $V_h\supset V_{2h}\supset V_{4h}$ denote the lowest-order N{\'e}delec spaces on
	three nested hexahedral meshes.  The mesh size is doubled in each coordinate
	direction between consecutive levels.  On a level
	$\ell\in\{h,2h,4h\}$, the discretization of the PML problem is written
	\begin{equation}
		A_\ell=C_\ell-k_0^2M_\ell,
		\label{eq:physical-level-operator}
	\end{equation}
	where $C_\ell$ and $M_\ell$ are, respectively, the discrete curl--curl and mass
	operators.  Each $A_\ell$ is obtained by rediscretizing the same variational
	problem with the blended rule of Section~\ref{sec:blended-quadrature}; the coarse
	operators are therefore
	not defined by a Galerkin product with the fine-grid matrix.

	We write $P_h:V_{2h}\rightarrow V_h$ and
	$P_{2h}:V_{4h}\rightarrow V_{2h}$ for the canonical edge-element
	interpolations \cite{Chanaud2014}.  To represent these maps, let $\bm\psi_e$ be the basis
	function dual to the edge moment $\ell_e$ and use the equivalent coordinates
	\begin{equation}
		U_e=|e|^{-1}\ell_e(\bm E),
		\qquad
		\bm\phi_e=|e|\bm\psi_e.
		\label{eq:edge-average-coordinates}
	\end{equation}
	Here $U_e$ is the component of the level coefficient vector associated with
	edge $e$.
	This diagonal change of basis leaves the finite element field unchanged and is
	used on all three levels.
	In these coordinates, prolongation copies the coefficient in the edge direction
	and interpolates linearly in the two transverse directions.  On the structured
	meshes considered here, we use the same symbols for these coordinate matrices
	and define the algorithmic restrictions by their Euclidean adjoints,
	$R_h=P_h^*$ and $R_{2h}=P_{2h}^*$, where $^*$ denotes conjugate transpose.
	Both maps are applied geometrically, without forming global transfer matrices.

	The auxiliary hierarchy is introduced only on the two coarser levels.  For
	$\ell\in\{2h,4h\}$, let
	\begin{equation}
		B_\ell
		=
		C_\ell-(1+\ii\sigma)k_0^2M_\ell,
		\qquad \sigma>0.
		\label{eq:shifted-level-operator}
	\end{equation}
	The added imaginary part damps the underresolved components that impede a
	standard coarse-grid iteration.  In particular, neither the fine-grid equation
	$A_h\bm U_h=\bm F_h$ nor the intermediate-grid equation involving $A_{2h}$ is
	replaced by its shifted counterpart.

	\subsection{Intermediate-grid correction}

	The intermediate-grid correction approximates the physical equation
	$A_{2h}\bm e_{2h}=\bm r_{2h}$ by a fixed-work inner FGMRES process.  Its
	preconditioner is one V-cycle for the shifted operators $B_{2h}$ and $B_{4h}$;
	the Krylov operator itself therefore remains unshifted.  To state this auxiliary
	cycle compactly, for $\ell\in\{2h,4h\}$ we let
	$\mathcal G_\ell(L,\bm g;\bm x_0)$ denote the output of a prescribed number of
	GMRES steps for $L\bm x=\bm g$, initialized by $\bm x_0$ and
	right-preconditioned by a fixed number of damped Jacobi sweeps.
	For a residual $\bm q_{2h}\in V_{2h}$, the shifted cycle is the map
	$\mathcal V_{2h}^{\sigma}:V_{2h}\rightarrow V_{2h}$ defined by
	\begin{equation}
	\begin{aligned}
		\bm y_{2h}
		&=\mathcal G_{2h}(B_{2h},\bm q_{2h};\bm 0),\\
		\bm q_{4h}
		&=R_{2h}\bigl(\bm q_{2h}-B_{2h}\bm y_{2h}\bigr),\\
		\bm e_{4h}
		&=\mathcal G_{4h}(B_{4h},\bm q_{4h};\bm 0),\\
		\mathcal V_{2h}^{\sigma}(\bm q_{2h})
		&=\mathcal G_{2h}
		\bigl(B_{2h},\bm q_{2h};
		\bm y_{2h}+P_{2h}\bm e_{4h}\bigr).
	\end{aligned}
		\label{eq:shifted-vcycle}
	\end{equation}
	Thus the bottom equation is also treated iteratively; no sparse factorization
	is required on the $4h$ grid.  Using $\mathcal V_{2h}^{\sigma}$ as a right
	preconditioner, we denote the fixed-work inner FGMRES approximation by
	\begin{equation}
		\bm e_{2h}
		=
		\mathcal K_{2h}
		\bigl(A_{2h},\bm r_{2h};\mathcal V_{2h}^{\sigma}\bigr)
		\label{eq:middle-krylov-map}
	\end{equation}
	where the iteration is initialized by zero.  Equation~\eqref{eq:middle-krylov-map} is the point
	at which the physical and auxiliary hierarchies interact: the residual and
	Krylov operator are unshifted, while the preconditioner for that Krylov process
	is shifted.

	\subsection{Three-grid cycle}

	The fine-grid pre- and post-smoothers use the analogous fixed-work map
	$\mathcal G_h(L,\bm g;\bm x_0)$.  The work counts and damping factors used for
	all three levels are given in Section~\ref{sec:experimental-protocol}.

	With the preceding notation, one application of the three-grid preconditioner
	$\mathcal M_h^{-1}$ is given below.  It combines smoothing on the
	physical fine-grid operator with an inexact correction from the physical
	intermediate-grid operator.

	\begin{center}
		\begin{minipage}{0.96\textwidth}
			\small
			\hrule
			\vspace{0.2em}
			\noindent\textbf{Algorithm 1} Application of the three-grid
			preconditioner $\bm z_h=\mathcal M_h^{-1}\bm r_h$.
			\vspace{0.2em}
			\hrule
			\vspace{0.15em}
			\begin{enumerate}
				\renewcommand{\labelenumi}{\arabic{enumi}:}
				\setlength{\itemsep}{0.12em}
				\setlength{\parsep}{0pt}
				\setlength{\parskip}{0pt}
				\item Pre-smooth on the fine grid:
				$\bm z_h^{\rm pre}=\mathcal G_h(A_h,\bm r_h;\bm 0)$.
				\item Restrict the fine-grid residual:
				$\bm r_{2h}=R_h(\bm r_h-A_h\bm z_h^{\rm pre})$.
				\item Compute the inexact intermediate-grid correction:
				$\bm e_{2h}=\mathcal K_{2h}
				(A_{2h},\bm r_{2h};\mathcal V_{2h}^{\sigma})$.
				\item Correct on the fine grid:
				$\widetilde{\bm z}_h=\bm z_h^{\rm pre}+P_h\bm e_{2h}$.
				\item Post-smooth on the fine grid:
				$\bm z_h=\mathcal G_h(A_h,\bm r_h;\widetilde{\bm z}_h)$.
			\end{enumerate}
			\vspace{0.05em}
			\hrule
		\end{minipage}
	\end{center}

	Because the fixed GMRES maps construct residual-dependent Krylov spaces,
	$\mathcal M_h^{-1}$ is not regarded as a stationary linear operator.  We
	therefore apply restarted FGMRES \cite{Saad1993} to the original system
	\begin{equation}
		A_h\bm U_h=\bm F_h,
		\label{eq:outer-unshifted-system}
	\end{equation}
	using $\mathcal M_h^{-1}$ as a right preconditioner.  Its restart length and
	stopping tolerance are external solver parameters specified in
	Sections~\ref{sec:numerical-validation} and~\ref{sec:experimental-protocol}.
	The outer stopping test and every residual reported below are consequently based on the unshifted
	operator $A_h$.  Equations~\eqref{eq:shifted-vcycle} and
	\eqref{eq:middle-krylov-map} also show why the construction remains fully
	matrix-free: every level requires only operator actions, diagonal Jacobi data,
	and geometric transfers, while the bottom correction is iterative rather than
	direct.

\section{Local Fourier Analysis}
\label{sec:maxwell-lfa}
We now employ a local Fourier analysis (LFA) to mathematically characterize the regularization supplied by the shifted auxiliary cycle. Throughout this section, the positive wavenumber $k_0>0$ is fixed. By lifting the gradient and transverse Maxwell branches to the harmonic space of the rediscretized hierarchy, we first derive exact symbol identities. A sampled finite-work analysis is then utilized to guide the selection of the shift and Jacobi damping parameters, clarifying the distinct spectral roles of the three grid levels.
	\subsection{Single-level Maxwell symbol}
	\label{sec:lfa-single-level}

	The multilevel analysis requires the symbol of the rediscretized Maxwell
	operator and its invariant branches.  Because the same discretization is used
	on all three levels, their symbols differ only through the mesh spacing, and it
	suffices to derive the $h$-grid symbol.  We consider the translation-invariant operator with
	$\epsilon_r=\mu_r=1$ on an infinite cubic grid.  For an edge orientation
	$r\in\{x,y,z\}$ and an index $\bm j\in\mathbb Z^3$, write an edge-average
	Fourier mode as
	\begin{equation}
		U_r(\bm j)=\widehat U_r
		\exp\!\left(\ii\bm\theta\cdot
		\left(\bm j+\tfrac12\bm e_r\right)\right),
		\qquad \bm\theta\in(-\pi,\pi]^3.
		\label{eq:lfa-edge-fourier-mode}
	\end{equation}
	Here $\bm e_r$ is the coordinate unit vector in the edge direction, and
	$\bm\theta$ is dimensionless.  Let $\tau\in[0,1]$ denote the blend
	parameter and define
	\begin{equation}
		a_\tau=\frac{2+\tau}{3},\qquad
		b_\tau=\frac{1-\tau}{3},\qquad
		q_j(\theta_j)=a_\tau+b_\tau\cos\theta_j,\qquad
		\delta_j=2\ii\sin(\theta_j/2).
		\label{eq:lfa-one-dimensional-symbols}
	\end{equation}
	The standard Gauss and Gauss--Lobatto rules correspond to $\tau=0$ and
	$\tau=1$, respectively.  The nonstandard rule in
	\eqref{eq:blended-rule-1d} corresponds to $\tau=1/2$, for which
	$q_j=(5+\cos\theta_j)/6$.  In the Fourier basis associated with the three edge
	orientations, the discrete curl symbol is
	\begin{equation}
		\widehat{\mathcal C}(\bm\theta)=
		\begin{bmatrix}
			0&-\delta_z&\delta_y\\
			\delta_z&0&-\delta_x\\
			-\delta_y&\delta_x&0
		\end{bmatrix}.
		\label{eq:lfa-curl-symbol}
	\end{equation}
	Here $\bm\delta=(\delta_x,\delta_y,\delta_z)^T$.  The complete edge symbol and
	its three Maxwell branches are given next.

	\begin{proposition}[Blended edge symbol]
	\label{prop:lfa-edge-symbol}
	Let $Q_f=\operatorname{diag}(q_x,q_y,q_z)$ and
	$Q_e=\operatorname{diag}(q_y q_z,q_x q_z,q_x q_y)$.  In the edge-average
	coordinates \eqref{eq:edge-average-coordinates},
	\begin{equation}
	\begin{aligned}
		\widehat C_h(\bm\theta)
		&=h\widehat{\mathcal C}(\bm\theta)^*Q_f
		\widehat{\mathcal C}(\bm\theta),\\
		\widehat M_h(\bm\theta)&=h^3Q_e,\\
		\widehat A_h^\sigma(\bm\theta)
		&=\widehat C_h(\bm\theta)
		-(1+\ii\sigma)k_0^2\widehat M_h(\bm\theta).
	\end{aligned}
		\label{eq:lfa-maxwell-symbol}
	\end{equation}
	Thus $\widehat A_\ell=\widehat A_\ell^0$ denotes the physical symbol, while
	$\widehat B_\ell=\widehat A_\ell^\sigma$ denotes the shifted auxiliary symbol
	for $\ell\in\{2h,4h\}$, consistently with
	\eqref{eq:physical-level-operator}--\eqref{eq:shifted-level-operator}.
	The generalized eigenvalues of $(\widehat C_h,\widehat M_h)$ are
	\begin{equation}
		0,\qquad \lambda_h(\bm\theta),\qquad \lambda_h(\bm\theta),
		\qquad
			\lambda_h(\bm\theta)
			=\frac{1}{h^2}\sum_{j=1}^3
			\frac{4\sin^2(\theta_j/2)}{q_j(\theta_j)}\geq0.
		\label{eq:lfa-maxwell-branches}
	\end{equation}
	Consequently, the generalized eigenvalues of
	$(\widehat A_h,\widehat M_h)$, with
	$\widehat A_h=\widehat A_h^0$, are $-k_0^2$ on the gradient branch and
	$\lambda_h(\bm\theta)-k_0^2$ on each of the two transverse branches.
	\end{proposition}

	\begin{proof}
	On an interval of length $h$, the blended local linear mass matrix is
	\[
		B_\tau=h\left\{
		\frac{1-\tau}{6}\begin{bmatrix}2&1\\1&2\end{bmatrix}
		+\frac{\tau}{2}I_2\right\}
		=h\begin{bmatrix}a_\tau/2&b_\tau/2\\b_\tau/2&a_\tau/2\end{bmatrix}.
	\]
	Assembly therefore gives the mass stencil
	$h[b_\tau/2,a_\tau,b_\tau/2]$, while the phase-centered incidence stencil is
	$[-1,1]$.  After factoring the powers of $h$, acting on a Fourier mode gives
	$a_\tau+b_\tau\cos\theta_j=q_j$ and
	$e^{\ii\theta_j/2}-e^{-\ii\theta_j/2}=\delta_j$.
	To make the tensor structure explicit, set $L_0(t)=(1-t)/2$ and
	$L_1(t)=(1+t)/2$ on $[-1,1]$.  Pulled back to the reference cube, the twelve
	scaled edge basis functions are the cyclic families
	\[
		L_\alpha(\widehat y)L_\beta(\widehat z)\bm e_x,\qquad
		L_\alpha(\widehat x)L_\beta(\widehat z)\bm e_y,\qquad
		L_\alpha(\widehat x)L_\beta(\widehat y)\bm e_z,\qquad
		\alpha,\beta\in\{0,1\}.
	\]
	For the affine cube map, $J=(h/2)I_3$.  The covariant transformation gives
	$\bm\psi_e\circ F=J^{-T}\widehat{\bm\psi}_e$, and multiplication by the physical
	edge length in \eqref{eq:edge-average-coordinates} yields
	$\bm\phi_e\circ F=2\widehat{\bm\psi}_e$.  Thus the scaled basis is dimensionless,
	its physical curl scales as $h^{-1}$, and the Jacobian contributes $h^3$ to a
	volume integral.  The mass and curl--curl blocks consequently scale as $h^3$
	and $h$, respectively.

	Thus an $r$-oriented basis function is constant in direction $r$ and linear in
	the other two.  Write $\widetilde B_\tau=B_\tau/h$ and let
	$\{r,s,t\}=\{x,y,z\}$.  Orthogonality of the coordinate vectors eliminates mass
	coupling between different edge orientations, and the four local $r$-edge
	functions have the tensor-product mass block
	\[
		M_K^{(r)}=h^3\bigl(\widetilde B_\tau^{(s)}\otimes
		\widetilde B_\tau^{(t)}\bigr).
	\]
	In each transverse direction, assembly of $\widetilde B_\tau$ has diagonal
	$a_\tau$, two neighboring entries $b_\tau/2$, and Fourier symbol $q_j$.
	Consequently,
	\[
		\widehat M_{h,rr}=h^3q_sq_t,
		\qquad
		\widehat M_{h,rs}=0\quad(r\ne s),
	\]
	which is $h^3Q_e$.

	The discrete curl maps edge coefficients to oriented face coefficients with
	incidence symbol $\widehat{\mathcal C}$.  For example,
	\[
		\nabla\times\bigl(L_\alpha(\widehat y)L_\beta(\widehat z)\bm e_x\bigr)
		=\frac{2}{h}\bigl(0,
		L_\alpha(\widehat y)L_\beta'(\widehat z),
		-L_\alpha'(\widehat y)L_\beta(\widehat z)\bigr).
	\]
	Since $2L_0'=-1$ and $2L_1'=1$, the nonzero components are the oriented
	face functions $\bm\chi_{r,\alpha}=h^{-1}L_\alpha(\widehat r)\bm e_r$.
	With $\sum_{\widehat q_j}w_j=2$, their local quadrature block is
	\begin{equation*}
	\begin{aligned}
		\langle\bm\chi_{r,\alpha},\bm\chi_{r,\beta}\rangle_{K,\tau}
		&=\frac{h^3}{8}h^{-2}(2)(2)
		\sum_{\widehat q_r}w_rL_\alpha(\widehat q_r)L_\beta(\widehat q_r)\\
		&=h(\widetilde B_\tau^{(r)})_{\alpha\beta}.
	\end{aligned}
	\end{equation*}
	The cyclic families reproduce the signs in $\widehat{\mathcal C}$.
	Assembly therefore gives the diagonal face-mass symbol
	$h\operatorname{diag}(q_x,q_y,q_z)=hQ_f$.  Composing the edge-to-face incidence,
	this block, and the adjoint incidence yields
	$h\widehat{\mathcal C}^{*}Q_f\widehat{\mathcal C}$, including the orientation
	signs fixed by \eqref{eq:lfa-curl-symbol}.  Thus tensor-product assembly gives
	\eqref{eq:lfa-maxwell-symbol}.  Since
	$q_j\geq(1+2\tau)/3>0$, the mass symbol is positive definite.  Set
	$D_e=Q_e^{1/2}$ and
	$X=Q_f^{1/2}\widehat{\mathcal C}(\bm\theta)D_e^{-1}$.  Entrywise,
	$X=\ii[\bm a]_{\times}$ is the cross-product matrix associated with
	\[
		\bm a=-\ii\left(\frac{\delta_x}{\sqrt{q_x}},
		\frac{\delta_y}{\sqrt{q_y}},
		\frac{\delta_z}{\sqrt{q_z}}\right)^T\in\mathbb R^3.
	\]
	Consequently, $X^*X=|\bm a|^2I_3-\bm a\bm a^T$ and
	\[
		\widehat M_h^{-1/2}\widehat C_h\widehat M_h^{-1/2}
		=h^{-2}X^*X.
	\]
	This matrix has eigenvalues $0$, $h^{-2}|\bm a|^2$, and
	$h^{-2}|\bm a|^2$.  If $\bm y$ is an eigenvector, then
	$\bm v=\widehat M_h^{-1/2}\bm y$ is the corresponding generalized eigenvector
	of $(\widehat C_h,\widehat M_h)$.  Its null vector is parallel to $\bm a$;
	because $Q_e^{-1/2}\bm a$ is a scalar multiple of $\bm\delta$, the generalized
	nullspace is the discrete gradient branch.  The Euclidean complement of
	$\bm a$ maps to the $\widehat M_h$-orthogonal complement of that branch and
	therefore gives the two transverse polarizations.  At $\bm\theta=\bm 0$,
	$\widehat C_h=0$ and all three
	branches coalesce at zero; \eqref{eq:lfa-maxwell-branches} then follows by
	direct evaluation (or continuity), although no polarization is distinguished at
	that single frequency.
	\end{proof}

	Equations~\eqref{eq:lfa-maxwell-symbol}--\eqref{eq:lfa-maxwell-branches}
	provide the single-level input required below.  They do not yet
	describe a coarse-grid correction, because coarsening couples several fine-grid
	frequencies to one coarse-grid mode.  We construct that coupled representation
	next.

	\subsection{Coarse-grid correction}
	\label{sec:lfa-harmonic-correction}

	To represent the coarse correction, let $\bm\theta$ satisfy
	$-\pi/2<\theta_j\leq\pi/2$, let
	$\bm\theta^{\bm\alpha}=\bm\theta+\pi\bm\alpha$ with
	$\bm\alpha\in\{0,1\}^3$, where frequencies are understood modulo $2\pi$.
	Coarsening by a factor of two identifies these eight aliases with the same
	coarse frequency $2\bm\theta$.  Including three edge orientations gives the
	$24$-dimensional invariant space in the following result.

	\begin{proposition}[Edge-harmonic closure]
	\label{prop:lfa-edge-harmonics}
	For every low frequency $\bm\theta$, the linear level operators and geometric
	transfers close on a $24$-dimensional fine-grid edge-harmonic space coupled to
	one three-component coarse mode.  The fine-grid operator on this space is
	\begin{equation}
		\widehat{\bm A}_h(\bm\theta)
		=\operatorname{diag}_{\bm\alpha\in\{0,1\}^3}
		\widehat A_h(\bm\theta^{\bm\alpha}).
		\label{eq:lfa-harmonic-maxwell-block}
	\end{equation}
	Set $c_j=\cos(\theta_j/2)$ and $s_j=\sin(\theta_j/2)$.  In the edge-centered
	Fourier basis \eqref{eq:lfa-edge-fourier-mode}, the prolongation block associated
	with alias $\bm\alpha$ is diagonal in edge orientation.  In the edge-average
	coordinates \eqref{eq:edge-average-coordinates}, its entries are
	\begin{equation}
		p_r^{\bm\alpha}(\bm\theta)
		=c_r^{1-\alpha_r}s_r^{\alpha_r}
		\prod_{j\ne r}c_j^{2(1-\alpha_j)}s_j^{2\alpha_j},
		\qquad r\in\{x,y,z\}.
		\label{eq:lfa-edge-transfer-symbol}
	\end{equation}
	Stacking the eight $3\times3$ blocks gives
	$\widehat{\bm P}_h\in\mathbb C^{24\times3}$ and, with the present Fourier
	normalization, $\widehat{\bm R}_h=8\widehat{\bm P}_h^*$.
	\end{proposition}

	\begin{proof}
	Translation invariance makes the eight aliases invariant under the fine-grid
	operator, which gives the block diagonal matrix
	\eqref{eq:lfa-harmonic-maxwell-block}.  The same argument on the coarse grid
	leaves the three-component mode at frequency $2\bm\theta$ invariant under its
	$3\times3$ symbol $\widehat A_{2h}(2\bm\theta)$.  In one dimension, a coarse mode at
	frequency $2\theta_j$ is decomposed into the fine aliases $\theta_j$ and
	$\theta_j+\pi$.  Along an edge direction, the coarse edge centered at
	$2n+1$ in fine-grid units is copied to the child edges centered at
	$2n+1/2$ and $2n+3/2$.  In the edge-centered basis
	\eqref{eq:lfa-edge-fourier-mode}, the two alias coefficients $a$ and $b$
	therefore satisfy
	\[
		a+\ii b=e^{\ii\theta_j/2},
		\qquad
		a-\ii b=e^{-\ii\theta_j/2},
	\]
	and hence $a=c_j$ and $b=s_j$.  For linear nodal interpolation, the value at
	a fine even point $2n$ is $e^{2\ii n\theta_j}$, whereas the value at the odd
	point $2n+1$ is
	\[
		\tfrac12\bigl(e^{2\ii n\theta_j}+e^{2\ii(n+1)\theta_j}\bigr)
		=\cos\theta_j\,e^{\ii(2n+1)\theta_j}.
	\]
	The two aliases agree at even points and have opposite signs at odd points.
	Thus $a+b=1$ and $a-b=\cos\theta_j$, so $a=c_j^2$ and $b=s_j^2$.  For an
	$r$-oriented edge, interpolation copies in direction
	$r$ and is linear in the two transverse directions.  Multiplying the three
	one-dimensional factors therefore gives \eqref{eq:lfa-edge-transfer-symbol},
	and stacking the eight aliases and three orientations gives
	$\widehat{\bm P}_h$ in the same edge-centered basis as the operator block.

	Restriction commutes with translations by one coarse cell.  Under such a
	translation every fine alias acquires the same phase
	$e^{2\ii\theta_j}$, because
	$e^{2\ii(\theta_j+\pi\alpha_j)}=e^{2\ii\theta_j}$.  On the coarse grid, the
	joint eigenspace of the three unit-cell translations with characters
	$(e^{2\ii\theta_x},e^{2\ii\theta_y},e^{2\ii\theta_z})$ is precisely the
	three-orientation Fourier mode at $2\bm\theta$.  Since restriction commutes with
	each translation, it preserves these three characters and therefore maps the
	whole fine-harmonic space into that coarse eigenspace.  Let $F$ and $G$ collect
	the unnormalized fine-harmonic and coarse
	Fourier basis vectors.  The coefficient symbols are defined by
	$P_hG=F\widehat{\bm P}_h$ and $R_hF=G\widehat{\bm R}_h$.  Therefore
	$R_h=P_h^*$ implies
	\[
		(G^*G)\widehat{\bm R}_h
		=\widehat{\bm P}_h^*(F^*F).
	\]
	To evaluate the Gram matrices, take
	$\Lambda_f=\prod_{j=1}^3\{0,\ldots,N_j-1\}$ with every $N_j$ even and set
	$n_f=N_1N_2N_3$.  The corresponding coarse grid has $N_j/2$ cells in direction
	$j$, and therefore $n_c=n_f/8$.  Different edge orientations have disjoint
	coefficient supports, while for a fixed orientation
	\[
		\langle F_{r,\bm\alpha},F_{r,\bm\beta}\rangle
		=e^{\ii\pi(\beta_r-\alpha_r)/2}
		\prod_{j=1}^3\sum_{m=0}^{N_j-1}
		e^{\ii\pi(\beta_j-\alpha_j)m}
		=n_f\delta_{\bm\alpha\bm\beta}.
	\]
	Indeed, each one-dimensional sum is $N_j$ when $\alpha_j=\beta_j$ and is
	$\sum_{m=0}^{N_j-1}(-1)^m=0$ otherwise; the edge-center phase therefore does
	not affect orthogonality.  The same calculation on the coarse grid gives
	$\langle G_r,G_s\rangle=n_c\delta_{rs}$.  Hence
	$F^*F=n_f I_{24}$ and $G^*G=n_c I_3$.  Since $n_f/n_c=8$, substitution
	in the preceding identity yields
	$\widehat{\bm R}_h=8\widehat{\bm P}_h^*$.
	\end{proof}

	The factor eight reflects only the relative normalization of the fine and coarse
	Fourier bases; globally, $R_h=P_h^*$.  For standard Gauss quadrature this
	normalization gives the exact nested identity
	\begin{equation}
		\widehat{\bm R}_h(\bm\theta)
		\widehat{\bm A}_h(\bm\theta)
		\widehat{\bm P}_h(\bm\theta)
		=\widehat A_{2h}(2\bm\theta).
		\label{eq:lfa-gauss-galerkin-identity}
	\end{equation}
	The blended coarse operator is instead rediscretized and need not equal this
	Galerkin product.

	For the tensor blended rule, the isotropic diagonal of $A_h^\sigma$ is
	\begin{equation}
		d_h^\sigma=4a_\tau h-(1+\ii\sigma)k_0^2h^3a_\tau^2.
		\label{eq:lfa-jacobi-diagonal}
	\end{equation}
	At a frequency $\bm\theta$, the symbol of one damped Jacobi sweep is
	\[
		\widehat J_h^\sigma(\bm\theta;\omega_h^{\mathrm{J}})
		=I_3-\omega_h^{\mathrm{J}}(d_h^\sigma)^{-1}
		\widehat A_h^\sigma(\bm\theta).
	\]
	Consequently, the prescribed $\nu_h$ sweeps define the linear approximate inverse
	\begin{equation}
		\widehat{\mathcal J}_{h,\nu_h}^{-1}
		(\bm\theta;\omega_h^{\mathrm{J}},\sigma)
		=\sum_{r=0}^{\nu_h-1}
		\widehat J_h^\sigma(\bm\theta;\omega_h^{\mathrm{J}})^r
		\omega_h^{\mathrm{J}}(d_h^\sigma)^{-1}I_3.
		\label{eq:lfa-jacobi-inverse}
	\end{equation}
	Suppressing the common base frequency and using bold symbols for the
	eight-alias blocks, set
	$\widehat{\bm J}_h^{\rm pre}=\widehat{\bm J}_h^{\rm post}
	=\operatorname{diag}_{\bm\alpha}
	\widehat J_h^0(\bm\theta^{\bm\alpha};\omega_h^{\mathrm{J}})^{\nu_h}$.
	With $\sigma=0$, stationary Jacobi in place of GMRES--Jacobi, and an exact
	rediscretized $2h$ correction, the idealized error symbol at frequencies where
	$\widehat A_{2h}(2\bm\theta)$ is nonsingular is
	\begin{equation}
		\widehat{\mathcal E}_h^{\rm TG}
		=\widehat{\bm J}_h^{\rm post}
		\left[I_{24}-\widehat{\bm P}_h\widehat A_{2h}(2\bm\theta)^{-1}
		\widehat{\bm R}_h\widehat{\bm A}_h\right]
		\widehat{\bm J}_h^{\rm pre}.
		\label{eq:lfa-fine-two-grid}
	\end{equation}
	If $\widehat{\mathcal E}_h^{\rm TG}
	=I_{24}-\widehat\Phi_h^{-1}\widehat{\bm A}_h$, then
	$\widehat\Phi_h^{-1}\widehat{\bm A}_h$ and the right-preconditioned symbol
	$\widehat{\bm A}_h\widehat\Phi_h^{-1}$ have the same spectrum.  The implemented
	method replaces the exact $2h$ inverse in \eqref{eq:lfa-fine-two-grid} by the
	fixed-work shifted $2h$--$4h$ process analyzed next.

	\subsection{Shifted auxiliary hierarchy}
	\label{sec:lfa-shifted-hierarchy}

	The three-grid cycle approximates the physical $2h$ inverse by fixed-work FGMRES
	preconditioned through $B_{2h}$.  The following result isolates the exact
	branchwise action of the shifted map
	$\widehat A_\ell\widehat B_\ell^{-1}$.

	\begin{proposition}[Exact shifted Maxwell map]
	\label{prop:lfa-shift-circle}
	Fix $\ell\in\{2h,4h\}$ and suppress a fixed frequency $\bm\theta$ in the symbols.  Let
	$\widehat C_\ell\bm v=\lambda\widehat M_\ell\bm v$ with
	$\lambda\geq0$, and set $\bm w=\widehat M_\ell\bm v$.  For
	$\widehat A_\ell=\widehat C_\ell-k_0^2\widehat M_\ell$ and
	$\widehat B_\ell=\widehat C_\ell-
	(1+\ii\sigma)k_0^2\widehat M_\ell$, $\sigma>0$, the right-preconditioned
	symbol $\widehat A_\ell\widehat B_\ell^{-1}$ has eigenvector $\bm w$ and
	eigenvalue
	\begin{equation}
		\mathfrak{z}_\sigma(\lambda)=
		\frac{\lambda-k_0^2}{\lambda-(1+\ii\sigma)k_0^2},
		\qquad
		\left(\operatorname{Re}\mathfrak{z}_\sigma(\lambda)-\frac12\right)^2
		+\left(\operatorname{Im}\mathfrak{z}_\sigma(\lambda)\right)^2=\frac14.
		\label{eq:lfa-shift-circle}
	\end{equation}
	Moreover,
	\begin{equation}
		\left|\lambda-(1+\ii\sigma)k_0^2\right|
		\geq \sigma k_0^2,
		\label{eq:lfa-shift-denominator-bound}
	\end{equation}
	so the shifted auxiliary symbol has no real-frequency singularity.  The image
	of the right-preconditioned physical symbol is bounded and lies in the closed
	right half-plane.  The
	discrete gradient branch is mapped to $(1+\ii\sigma)^{-1}$, while both
	transverse branches lie on the same circle.
	\end{proposition}

	\begin{proof}
	By Proposition~\ref{prop:lfa-edge-symbol}, $\widehat M_\ell$ is positive
	definite and $\widehat C_\ell$ is Hermitian positive semidefinite.  Their
	generalized eigenvectors therefore form a basis, and
	$\bm w=\widehat M_\ell\bm v\ne\bm0$.  For every generalized eigenvalue
	$\lambda\geq0$, the standing assumption $\omega>0$ in
	Section~\ref{sec:maxwell-discretization} gives $k_0>0$, and
	\[
		|\lambda-(1+\ii\sigma)k_0^2|^2
		=(\lambda-k_0^2)^2+\sigma^2k_0^4
		\geq\sigma^2k_0^4>0.
	\]
	Thus $|\lambda-(1+\ii\sigma)k_0^2|\geq\sigma k_0^2$.  If
	$\{\bm v_j\}$ is a generalized eigenbasis, then
	$\{\widehat M_\ell\bm v_j\}$ is also a basis and
	\[
		\widehat B_\ell\bm v_j
		=\bigl(\lambda_j-(1+\ii\sigma)k_0^2\bigr)
		\widehat M_\ell\bm v_j.
	\]
	Every scalar factor is nonzero, so $\widehat B_\ell$ maps a basis to a basis and
	is invertible.
	The generalized eigenrelation implies
	\[
		\widehat A_\ell\bm v=(\lambda-k_0^2)\widehat M_\ell\bm v,
		\qquad
		\widehat B_\ell\bm v=
		\bigl(\lambda-(1+\ii\sigma)k_0^2\bigr)
		\widehat M_\ell\bm v.
	\]
	The second relation gives
	$\widehat B_\ell^{-1}\bm w=
	\bm v/[\lambda-(1+\ii\sigma)k_0^2]$.  Applying $\widehat A_\ell$ then shows
	that $\bm w$ is an eigenvector of
	$\widehat A_\ell\widehat B_\ell^{-1}$ with eigenvalue
	$\mathfrak{z}_\sigma(\lambda)$.  Set
	$t=\lambda-k_0^2$ and $\beta=\sigma k_0^2$.  Then
	\[
		\mathfrak{z}_\sigma(\lambda)=\frac{t}{t-\ii\beta}
		=\frac{t^2+\ii\beta t}{t^2+\beta^2},
	\]
	so $\operatorname{Re}\mathfrak{z}_\sigma(\lambda)=t^2/(t^2+\beta^2)\geq0$
	and $\operatorname{Im}\mathfrak{z}_\sigma(\lambda)=\beta t/(t^2+\beta^2)$.
	Consequently,
	\[
		\left(\operatorname{Re}\mathfrak{z}_\sigma-\tfrac12\right)^2
		+\left(\operatorname{Im}\mathfrak{z}_\sigma\right)^2
		=\frac{(t^2-\beta^2)^2+4\beta^2t^2}
		{4(t^2+\beta^2)^2}=\frac14.
	\]
	This identity also proves boundedness and the closed-right-half-plane claim.
	Finally, Proposition~\ref{prop:lfa-edge-symbol} identifies $\lambda=0$ as the
	gradient branch, which is mapped to $(1+\ii\sigma)^{-1}$; for nonzero frequency,
	the repeated nonzero eigenvalue gives the two transverse branches.
	\end{proof}

	The denominator bound removes real-frequency singularities from the auxiliary
	inverse, while the circle identity confines the exactly preconditioned spectrum.
	Modes near the physical dispersion surface remain near the origin because the
	numerator is unshifted.  The practical $2h$--$4h$ cycle approximates
	$B_{2h}^{-1}$; at the bottom level its GMRES operator is
	$\widehat B_{4h}(\bm\theta)
	\widehat{\mathcal J}_{4h,\nu_{4h}}^{-1}
	(\bm\theta;\omega_{4h}^{\mathrm{J}},\sigma)$, with
	$\widehat{\mathcal J}_{4h,\nu_{4h}}^{-1}$ defined by
	\eqref{eq:lfa-jacobi-inverse}.  We now use these symbols to select the shift and
	damping parameters of the finite-work cycle.

	\subsection{Parameter selection}
	\label{sec:lfa-parameter-guidance}

	Only $\sigma$ and the Jacobi damping factors
	$\omega_\ell^{\mathrm{J}}$ are examined here; the Krylov and sweep counts remain
	the prescribed budgets in Section~\ref{sec:experimental-protocol}.  Because a
	fixed-step Krylov polynomial
	depends on its incoming residual, the complete cycle has no single
	residual-independent error symbol.  We therefore combine the exact identities
	above with a sampled finite-work diagnostic
	\cite{WienandsOosterleeWashio2000,WienandsOosterlee2001}, evaluated at the
	minimum experimental resolution $\mathrm{ppw}=2\pi/(k_0h)=8$.  The scan
	uses no iteration counts, timings, or heterogeneous data from
	Sections~\ref{sec:numerical-validation}--\ref{sec:performance}.

	\paragraph{Stationary level maps}
	For the GMRES--Jacobi components, set
	$\sigma_h=0$ and $\sigma_{2h}=\sigma_{4h}=\sigma$.  Let the level spacings be
	$h_h=h$, $h_{2h}=2h$, and $h_{4h}=4h$, and set
	\[
		d_\ell^{\sigma_\ell}
		=4a_\tau h_\ell
		-(1+\ii\sigma_\ell)k_0^2h_\ell^3a_\tau^2.
	\]
	Let $\widehat{\mathcal J}_{\ell,\nu_\ell}^{-1}
	(\bm\theta;\omega_\ell^{\mathrm{J}},\sigma_\ell)$ denote the levelwise analogue
	of \eqref{eq:lfa-jacobi-inverse}.  If $\lambda_J$ is an eigenvalue of
	$(d_\ell^{\sigma_\ell})^{-1}
	\widehat A_\ell^{\sigma_\ell}(\bm\theta)$, the corresponding
	eigenvalue of
	$\widehat A_\ell^{\sigma_\ell}(\bm\theta)
	\widehat{\mathcal J}_{\ell,\nu_\ell}^{-1}
	(\bm\theta;\omega_\ell^{\mathrm{J}},\sigma_\ell)$ is
	\begin{equation}
		g_{\omega_\ell^{\mathrm{J}},\nu_\ell}(\lambda_J)
		=1-(1-\omega_\ell^{\mathrm{J}}\lambda_J)^{\nu_\ell}.
		\label{eq:lfa-jacobi-map}
	\end{equation}
	Since $d_\ell^{\sigma_\ell}$ is scalar on the homogeneous cubic grid,
	\eqref{eq:lfa-jacobi-map} follows by summing the stationary polynomial.  It
	describes the spectrum presented to GMRES, not its residual-dependent polynomial.

	\paragraph{Complete fixed-work diagnostic}
	Applying Proposition~\ref{prop:lfa-edge-harmonics} recursively couples $4^3$
	fine-grid aliases for each base frequency in $(-\pi/4,\pi/4]^3$, giving blocks of
	dimensions $192$, $24$, and $3$ on the three levels.  Let
	$\bm\vartheta=(\sigma,\omega_h^{\mathrm{J}},\omega_{2h}^{\mathrm{J}},
	\omega_{4h}^{\mathrm{J}})$.  We write
	$\widehat{\mathcal M}_h^{-1}(\bm\theta;\bm\vartheta)[\bm r]$ for the output of
	one application of the prescribed three-grid process to a harmonic residual
	$\bm r$.  On the tensor-product midpoint grid
	\[
		\Theta_{n_\theta}
		=
		\left\{-\frac{\pi}{4}
		+\frac{(s+1/2)\pi}{2n_\theta}:
		s=0,\ldots,n_\theta-1\right\}^{3}
	\]
	and a fixed set $\mathcal R_{n_r}\subset\mathbb C^{192}$ of normalized complex
	residual probes, define
	\begin{equation}
		\widetilde\Gamma_{\rm 3G}^{n_\theta,n_r}(\bm\vartheta)=
		\max_{\bm\theta\in\Theta_{n_\theta}}
		\max_{\bm r\in\mathcal R_{n_r}}
		\frac{\left\|\bm r-
		\widehat{\mathbb A}_h(\bm\theta)
		\widehat{\mathcal M}_h^{-1}(\bm\theta;\bm\vartheta)[\bm r]
		\right\|_2}{\|\bm r\|_2},
		\label{eq:lfa-complete-cycle-factor}
	\end{equation}
	where $\widehat{\mathbb A}_h$ contains all $4^3$ fine aliases.  This notation
	distinguishes the $192\times192$ two-coarsening closure from
	$\widehat{\bm A}_h$ in \eqref{eq:lfa-harmonic-maxwell-block}.  We compare the
	maximum, $95$th percentile, and mean sampled reductions, refining to $8^3$
	frequencies and twelve probes per frequency.  The quantity in
	\eqref{eq:lfa-complete-cycle-factor} is a finite diagnostic, not a uniform bound.

	\paragraph{Damping selection}
	We scan one damping factor at a time in
	\eqref{eq:lfa-complete-cycle-factor}, holding the remaining parameters fixed.
	On the finest sample, the worst-case fine-grid indicator is minimized at
	$\omega_h^{\mathrm{J}}=0.55$; the upper-tail and mean minima occur at $0.575$.
	We select $0.55$.  Over $0.4\leq\omega_{2h}^{\mathrm{J}}\leq0.8$, the maximum
	changes by less than $0.7\%$ and the other indicators by less than $0.3\%$;
	$\omega_{2h}^{\mathrm{J}}=0.55$ lies in this insensitive range.  For
	$m_{4h}\geq3$, nonsingular three-dimensional bottom blocks are solved exactly by
	GMRES in exact arithmetic, making the local prediction independent of
	$\omega_{4h}^{\mathrm{J}}$; we fix it at $0.4$.

	\paragraph{Shift selection}
	The shift balances two competing effects.  Proposition~\ref{prop:lfa-shift-circle}
	shows that increasing $\sigma$ improves the lower bound for the auxiliary
	denominator.  On the other hand,
	\begin{equation}
		|1-\mathfrak{z}_\sigma(\lambda)|
		=
		\frac{\sigma k_0^2}
		{\sqrt{(\lambda-k_0^2)^2+\sigma^2k_0^4}},
		\label{eq:lfa-shift-perturbation}
	\end{equation}
	which increases with $\sigma$ for every fixed off-resonant branch value.
	Thus increasing $\sigma$ regularizes the auxiliary inverse but moves it farther
	from the physical inverse.  On the same frequency grid, we separately measure
	the residual factors of one shifted $B_{2h}$ V-cycle and of the prescribed
	$A_{2h}$ FGMRES process preconditioned by that cycle.  The former decreases with
	$\sigma$, whereas the latter eventually increases; their sampled worst cases
	balance near $\sigma=0.3$.  The complete-cycle response is flatter:
	under refinement, its shallow minima remain in $0.2\leq\sigma\leq0.45$, within
	which every reported indicator varies by less than $0.7\%$.  We choose
	$\sigma=0.4$, obtaining stronger denominator regularization while remaining on
	this plateau.  The selected fine damping is therefore supported by a stable
	sampled minimum, while the shift and remaining damping factors lie in ranges
	identified as insensitive.  These values are fixed for all experiments.

	\subsection{Three-grid spectral roles}
	\label{sec:lfa-spectral-roles}

	With these parameter roles established,
	Figure~\ref{fig:maxwell-lfa-spectra} shows how the spectral tasks
	are divided among the three levels.  It evaluates the stationary component maps with
	the Jacobi sweep counts reported in Section~\ref{sec:experimental-protocol} at
	fine-grid $\mathrm{ppw}=8$.  These are component, not complete-cycle, spectra.  In the left
	panel, an exact rediscretized $2h$
	correction with the prescribed pre- and post-Jacobi maps sends many eigenvalues
	close to one while leaving indefinite outliers, for which GMRES supplies a
	residual-adapted polynomial.  In the center, the
	unshifted intermediate-grid Jacobi map produces a broad real spectrum crossing the origin,
	whereas exact shifted preconditioning confines the eigenvalues to the circle in
	Proposition~\ref{prop:lfa-shift-circle}; the shifted V-cycle approximates this
	action inside the unshifted $2h$ FGMRES solve.  The bottom map remains separated
	from the origin, leaving a small fixed GMRES process to treat its spread.
	\begin{figure}[H]
		\centering
		\includegraphics[width=0.94\textwidth]{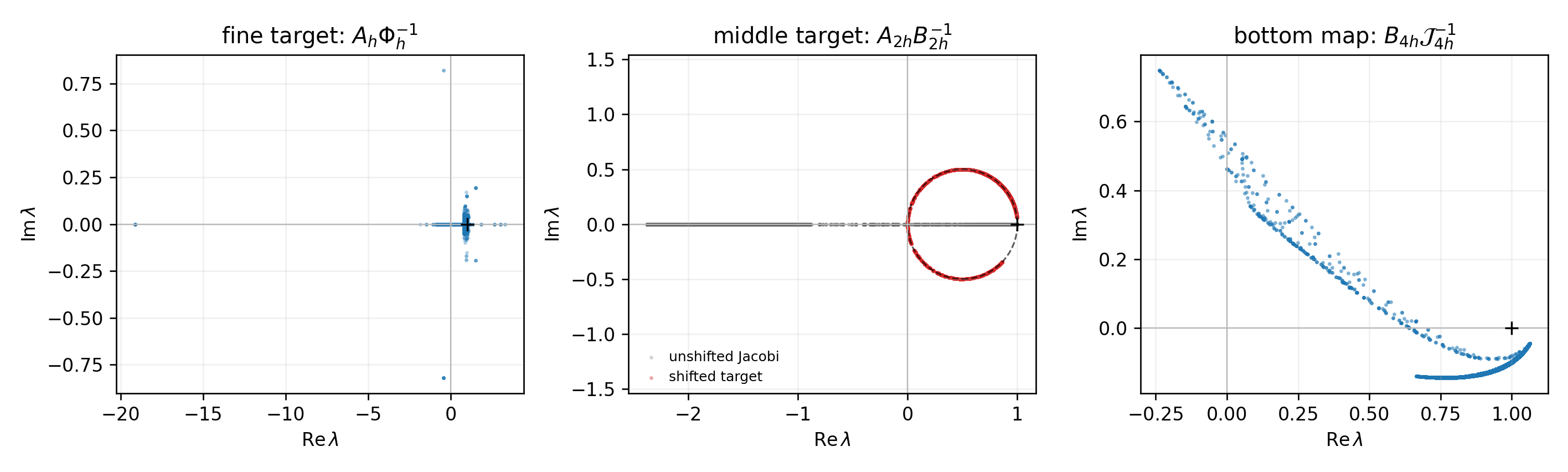}
		\caption{Component spectra at fine-grid $\mathrm{ppw}=8$ for
		$\tau=1/2$, $\sigma=0.4$, and the Jacobi sweep counts in
		Section~\ref{sec:experimental-protocol}.
		Left: the idealized unshifted fine-grid two-grid operator with an exact
		rediscretized coarse inverse, the prescribed stationary smoother, and
		$\omega_h^{\mathrm{J}}=0.55$.  Center: the unshifted intermediate-grid Jacobi map with
		$\omega_{2h}^{\mathrm{J}}=0.55$ (gray) and the exact shifted target (red),
		together with the circle in \eqref{eq:lfa-shift-circle}.  Right: the shifted
		$4h$ operator with its prescribed Jacobi preconditioner and
		$\omega_{4h}^{\mathrm{J}}=0.4$.}
		\label{fig:maxwell-lfa-spectra}
	\end{figure}

	An independent periodic assembly, formed by direct quadrature of the twelve
	local N{\'e}delec fields and curls rather than by the tensor symbol formulas,
	reproduced the local blocks, edge symbol, Maxwell branches, transfer identity,
	and Gauss Galerkin identity to relative accuracy $10^{-13}$ or better.  As
	expected, blended rediscretization did not satisfy the Gauss Galerkin identity.
	Variable coefficients and PML profiles lie outside this translation-invariant
	analysis and are examined numerically in
	Sections~\ref{sec:numerical-validation}--\ref{sec:performance}.
	
	\section{Numerical Validation}
	\label{sec:numerical-validation}
	
	Before considering heterogeneous media and large-scale calculations, we use a
	homogeneous point-source problem to establish the numerical and computational
	choices used in the remainder of the paper.  The outgoing Maxwell Green tensor
	provides an analytical reference against which the discrete point source, the
	blended lowest-order N{\'e}delec operator, and the PML truncation can be assessed
	together.  The same controlled problem is then used to identify a sufficient
	outer stopping tolerance, to verify the use of complex single precision, and to
	compare the CPU and GPU realizations of the three-grid algorithm.  Consequently,
	the parameters used in the large-scale experiments are selected from their effect
	on the computed field, rather than from timing alone.
	
	\subsection{Green-tensor validation}
	
	The physical domain is
	\[
	\Omega_{\rm ph}=[-1,1]^3.
	\]
	The medium is homogeneous, and the validation uses a $768^3$ total mesh,
	including the PML cells, with
	\[
	{\rm ppw}=8,\qquad n_{\rm pml}=12,
	\]
	where ${\rm ppw}$ is the number of fine-grid points per physical wavelength
	in the homogeneous domain and $n_{\rm pml}$ is the number of fine-grid cells in
	each PML layer.
	The physical domain therefore spans approximately 93 wavelengths.  A
	$z$-oriented point current is prescribed near $(0,0.5,0)$.  Since the discrete
	unknowns are edge moments, the source is not represented by assigning a value to
	a node or by treating it as a cellwise constant load.  Instead, the discrete load
	on $V_h$ is defined by evaluating
	$\bm v\mapsto\bm e_z\cdot\bm v(\bm x_s)$, where $\bm e_z$ is the unit vector
	in the $z$ direction, on the local N{\'e}delec basis functions
	of the source cell.  Writing $\bm x_s=(x_s,y_s,z_s)$, the corresponding
	discrete source location is
	\[
	\bm{x}_s=(0.00134333,\;0.50134295,\;0.00134333).
	\]
	At the cell center, the functional acts on the four $z$-oriented edge basis
	functions, up to the common source normalization introduced below.

	The reference, tolerance, and precision tests use the GPU matrix-free solver.
	In the CPU--GPU comparison, only the realization of the level operators is
	changed; the discrete system and preconditioner remain fixed.  In all cases,
	the outer equation is unshifted, and the shift is confined to the auxiliary
	hierarchy.
	
		For a homogeneous medium, let
	$k=k_0\sqrt{\epsilon_r\mu_r}$ denote the physical wavenumber.  The outgoing
	scalar Green function is
	\[
	G_k(\bm{x},\bm{x}_s)
	=
	\frac{\exp(\ii k|\bm{x}-\bm{x}_s|)}
	{4\pi |\bm{x}-\bm{x}_s|}.
	\]
		For a $z$-directed electric source, the reference $E_z$ component is computed
		from the Maxwell Green tensor \cite{Monk2003} as
	\[
	E_z^{\rm ref}
	=
	c_{\rm src}\left(1+\frac{1}{k^2}\partial_{zz}\right)
	G_k(\bm{x},\bm{x}_s),
	\]
	where the complex scalar $c_{\rm src}$ is determined by least squares over the
	retained physical-domain sampling points described below.  It accounts
	for the amplitude normalization introduced by applying a continuous point
	current through discrete edge moments; the phase and spatial variation of the
	reference field are not fitted.
	
	The continuous reference is singular at the source, where a direct comparison
	with any finite-dimensional load representation would be dominated by the
	source regularization.  Following the assessment procedure used for
	geophysical wave solvers \cite{Tournier2022}, we therefore exclude a
	source-centered ball of radius one wavelength,
	\[
	\mathcal B_{\lambda_{\rm w}}(\bm{x}_s)=
	\{\bm{x}:|\bm{x}-\bm{x}_s|<\lambda_{\rm w}\},\qquad
		\lambda_{\rm w}=\frac{2\pi}{k},
	\]
	and compensate for the geometric amplitude decay by a distance gain.  The
	reported error is
	\[
	\begin{aligned}
		{\rm Err}
		&=
		\frac{%
			\sum_{\bm{x}^{(j)}\in \Omega_{\rm ph}\setminus \mathcal B_{\lambda_{\rm w}}}
			|\bm{x}^{(j)}-\bm{x}_s|
			\left|\operatorname{Re}(E_{z,h}(\bm{x}^{(j)})-E_z^{\rm ref}(\bm{x}^{(j)}))\right|
		}{%
			\sum_{\bm{x}^{(j)}\in \Omega_{\rm ph}\setminus \mathcal B_{\lambda_{\rm w}}}
			|\bm{x}^{(j)}-\bm{x}_s|
			\left|\operatorname{Re}(E_z^{\rm ref}(\bm{x}^{(j)}))\right|
		}  \\
		&\quad+
		\frac{%
			\sum_{\bm{x}^{(j)}\in \Omega_{\rm ph}\setminus \mathcal B_{\lambda_{\rm w}}}
			|\bm{x}^{(j)}-\bm{x}_s|
			\left|\operatorname{Im}(E_{z,h}(\bm{x}^{(j)})-E_z^{\rm ref}(\bm{x}^{(j)}))\right|
		}{%
			\sum_{\bm{x}^{(j)}\in \Omega_{\rm ph}\setminus \mathcal B_{\lambda_{\rm w}}}
			|\bm{x}^{(j)}-\bm{x}_s|
			\left|\operatorname{Im}(E_z^{\rm ref}(\bm{x}^{(j)}))\right|
		},
	\end{aligned}
	\]
	where $E_{z,h}=\bm e_z\cdot\bm E_h$ and $\bm x^{(j)}$ are the sampling points in
	the physical domain; PML cells are not included in the sums.

	Table~\ref{tab:pointsource-validation} summarizes the solver statistics and
	the distance-gained error for this validation.  The numerical and analytical wavefronts in
	Figure~\ref{fig:pointsource-slices} remain aligned on all three source-crossing
	planes, while the visible discrepancy is concentrated near the discrete source.
	Figure~\ref{fig:pointsource-1d} makes the phase comparison more explicit on two
	receiver lines in the plane $z=z_s$, neither of which intersects the singularity.
	Multiplication by the source distance removes the dominant $1/r$ decay, and the
	numerical profiles closely follow the analytical oscillations across the physical
	domain.  Together, the comparisons check the source, discretization, and PML
		treatment under long-range propagation without attributing the remaining error to
		any one component.

	\begin{table}[H]
		\centering
		\caption{Homogeneous point-source validation on eight A100-40G GPUs.
			The true residual refers to the original unshifted system, and PC calls
			counts preconditioner applications.  Performance entries match the
			$768^3$ homogeneous run in Table~\ref{tab:a100-mesh-weak-scaling}.}
		\resizebox{\textwidth}{!}{%
		\begin{tabular}{cccccccccc}
			\toprule
			mesh & GPUs & ppw & $n_{\rm pml}$ & PC calls
			& $\|A_h\bm U_h-\bm F_h\|_2/\|\bm F_h\|_2$ & setup & solve & peak memory & Err \\
			\midrule
			$768^3$ & 8 & 8 & 12 & 30 & $8.2410\times 10^{-4}$
			& 0.221 s & 17.668 s & 296.8 GiB & $1.5786\times 10^{-1}$ \\
			\bottomrule
		\end{tabular}%
		}
		\label{tab:pointsource-validation}
	\end{table}

	\begin{figure}[!t]
		\centering
		\includegraphics[width=0.775\textwidth,trim=0 14bp 0 18bp,clip]{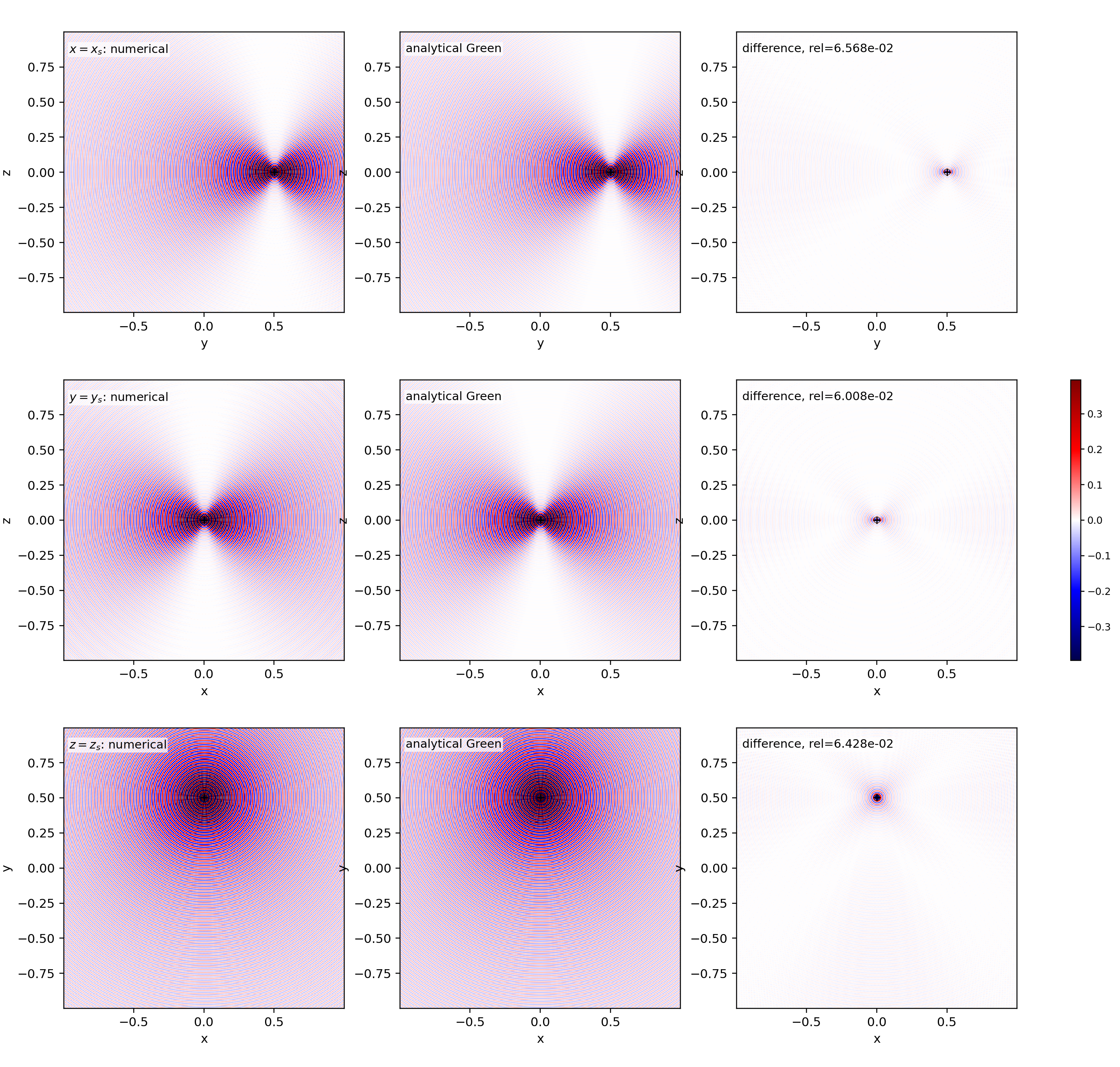}
		\caption{Homogeneous point-source validation on a $768^3$ total mesh.
			Rows show the source-crossing planes $x=x_s$, $y=y_s$, and $z=z_s$;
			columns show numerical $\operatorname{Re}(E_z)$, the fitted Maxwell Green
			reference, and their difference.  The cross and circle mark the source and
			the one-wavelength exclusion radius, respectively.}
		\label{fig:pointsource-slices}
	\end{figure}
	
	\begin{figure}[!t]
		\centering
		\includegraphics[width=0.835\textwidth,trim=4bp 6bp 4bp 7bp,clip]{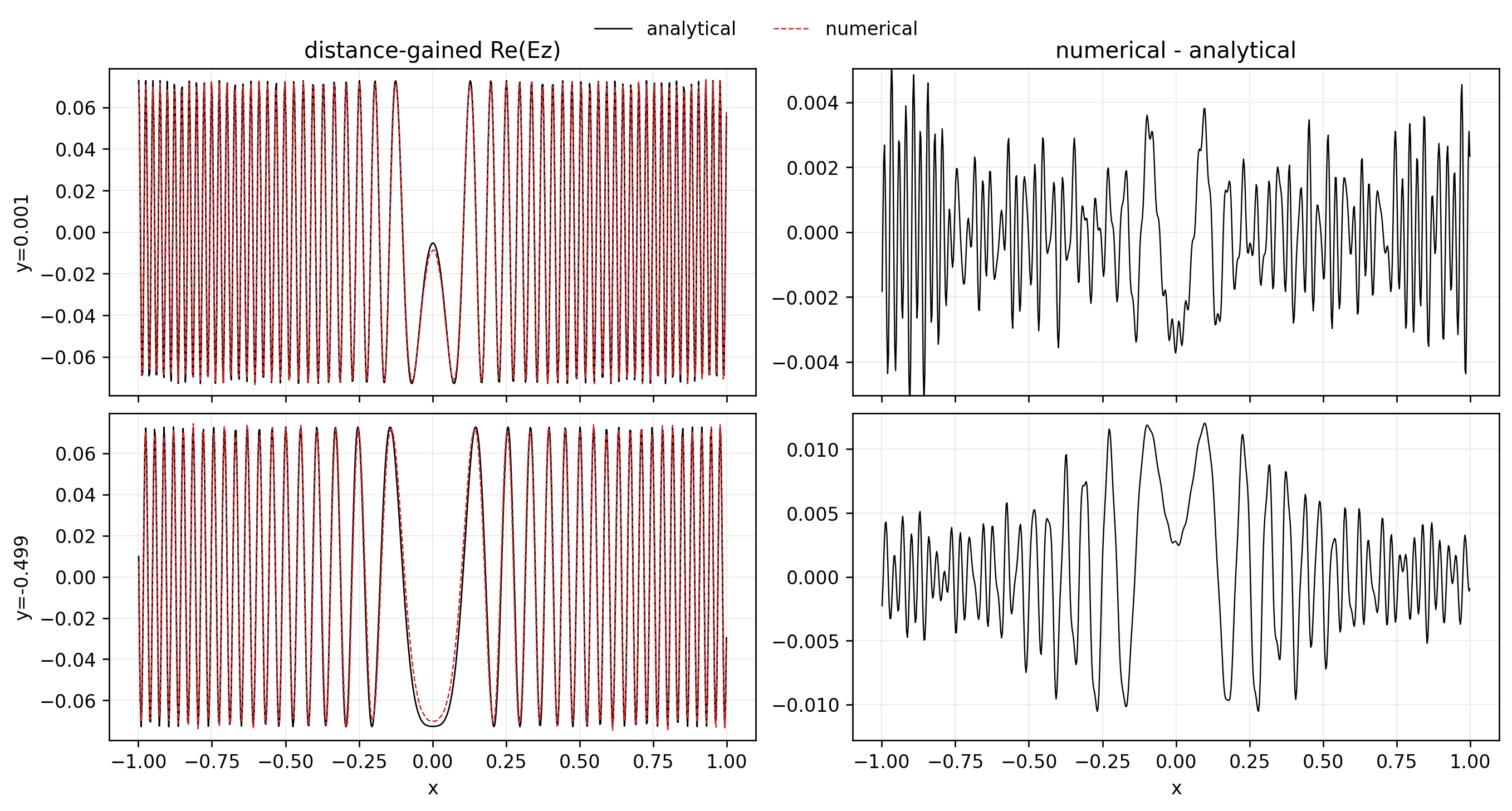}
		\caption{Receiver-line validation on the $768^3$ total mesh.  The left
			panels compare distance-gained numerical and analytical
			$\operatorname{Re}(E_z)$; the right panels show their difference.}
		\label{fig:pointsource-1d}
	\end{figure}
		
	\subsection{Stopping tolerance}
	
	The Green-tensor error contains both discretization error and algebraic error.
	To determine when the latter becomes negligible, we solve the same $768^3$
	problem with outer tolerances $10^{-2}$, $10^{-3}$, $10^{-4}$, and $10^{-5}$,
	changing no other parameter.  Table~\ref{tab:pointsource-tolerance} reports the
	relative error on the receiver line $y\approx0$, $z=z_s$, together with the
	change in that profile relative to the $10^{-5}$ solution.  Timings are omitted
	because the purpose here is to select an accuracy threshold, not to compare
	implementations.
	
	\begin{table}[H]
		\centering
		\caption{Tolerance study on the homogeneous point-source test.}
		\begin{tabular}{ccc}
			\toprule
			outer tol. & profile error & change vs. $10^{-5}$ \\
			\midrule
			$10^{-2}$ & $3.4800\times 10^{-2}$ & $1.6005\times 10^{-2}$ \\
			$10^{-3}$ & $3.6370\times 10^{-2}$ & $5.5525\times 10^{-3}$ \\
			$10^{-4}$ & $3.6959\times 10^{-2}$ & $3.0800\times 10^{-4}$ \\
			$10^{-5}$ & $3.6989\times 10^{-2}$ & 0 \\
			\bottomrule
		\end{tabular}
		\label{tab:pointsource-tolerance}
	\end{table}

	\begin{figure}[H]
		\centering
		\includegraphics[width=0.90\textwidth,trim=8bp 15bp 5bp 0,clip]{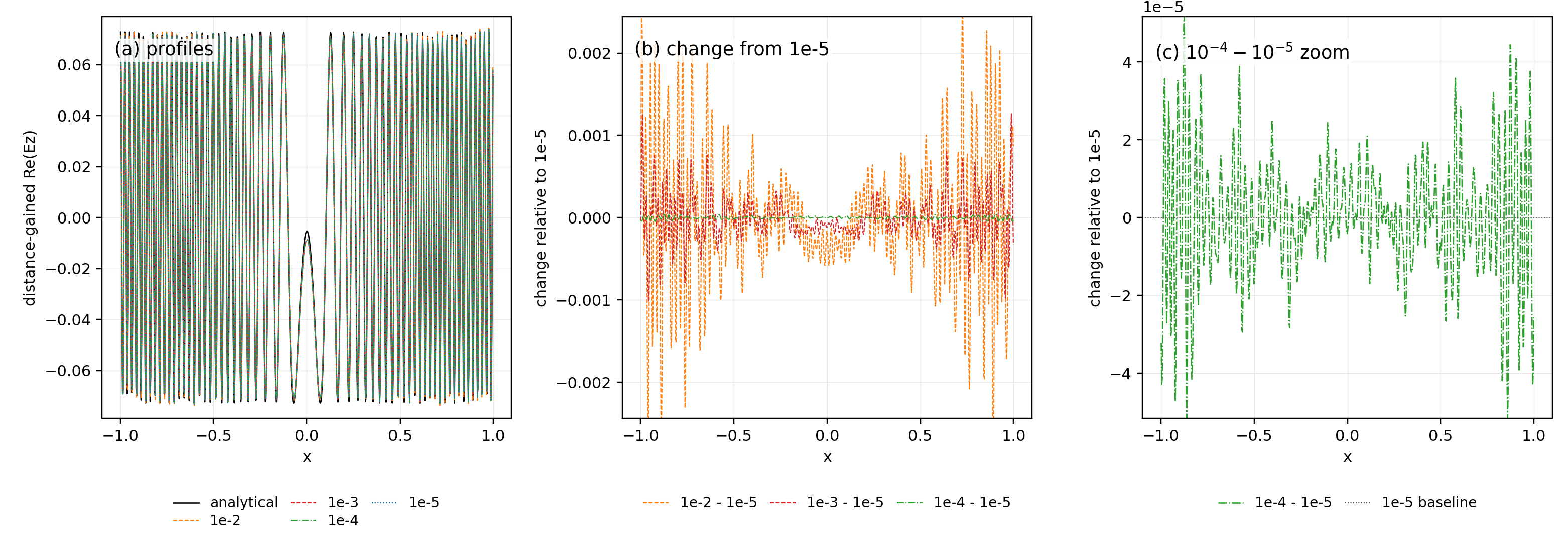}
		\caption{Tolerance effect on the source-crossing receiver profile.
			Panel (a) compares the analytical profile and all numerical profiles.
			Panel (b) subtracts the $10^{-5}$ solve.  Panel (c) zooms in on the
			$10^{-4}$--$10^{-5}$ difference.}
		\label{fig:pointsource-tolerance}
	\end{figure}
	\FloatBarrier
	
	At a tolerance of $10^{-3}$, the receiver-line error is
	$3.6370\times10^{-2}$, compared with $3.6989\times10^{-2}$ for the $10^{-5}$
	solution, and the profile changes by only $5.5525\times10^{-3}$.  Further
	reduction of the algebraic residual therefore has little effect on the field at
	the accuracy delivered by the present discretization and source treatment.  We
	use $10^{-3}$ in the subsequent performance experiments to avoid oversolving the
	discrete system.
	
	\subsection{Single-precision validation}
	
	Having fixed the outer tolerance, we compare complex single and double precision
	on the same $768^3$ point-source problem.  The discrete operator, source, PML,
	preconditioner, and stopping criterion are unchanged.  Only the arithmetic
	precision differs.
	Because the double-precision hierarchy does not fit on eight 40~GiB GPUs, the
	two runs both use 16 A100-40G GPUs.
	
	\begin{table}[H]
		\centering
		\caption{Single- and double-precision comparison for the homogeneous
			point-source problem.  The solution difference is measured over the full discrete
			solution vector, using the double-precision solution as the reference.}
		\begin{tabular}{cccc}
			\toprule
			precision & PC calls & $\|A_h\bm U_h-\bm F_h\|_2/\|\bm F_h\|_2$
			& solution difference \\
			\midrule
			single & 30 & $8.5032\times 10^{-4}$ & $3.84\times 10^{-5}$ \\
			double & 30 & $8.5031\times 10^{-4}$ & -- \\
			\bottomrule
		\end{tabular}%
		\label{tab:precision-check}
	\end{table}
	
	Both precisions require 30 preconditioner applications and attain the same true
	residual.  Their full-vector relative difference is
	$3.84\times10^{-5}$, well below the selected algebraic tolerance.  Complex
	single precision therefore preserves the numerical result and is used in the
	remaining GPU experiments.
	
	\subsection{CPU--GPU comparison}
	
	We finally compare the fastest CPU and GPU implementations developed in this
	work for the same discretized Maxwell problem.  Both use the lowest-order
	Nedelec space, blended quadrature, complex single precision, the same three-grid
	preconditioner, and the same outer stopping criterion.  The low-level operator
	representation is allowed to differ so that neither architecture is penalized by
	a kernel designed for the other.
	
	The CPU implementation used in the comparison is a stored row-stencil PETSc
	MatShell.  It does not form a generic PETSc AIJ matrix, but it explicitly stores
	the coefficients of every owned edge row and the halo columns required by that
	row; it is therefore a stored-coefficient, rather than a strict matrix-free,
	operator.  We implemented and timed both this version and a strict
	element-by-element matrix-free CPU MatShell.  The matrix-free version was
	substantially slower and was stopped before completion.  Its cost comes from reconstructing the local
	action and scattering contributions into shared edge unknowns at every operator
	application.  Storing the row stencil removes this repeated work at the price of
	a larger memory footprint.  The CPU results therefore use the faster
	stored-coefficient implementation.
	
	The corresponding trade-off is different on GPUs.  A cached Maxwell edge
	stencil requires streaming many complex coefficients from global memory for each
	unknown and is bandwidth limited.  The element-based GPU kernel instead retains
	the compact local N{\'e}delec data and recomputes the element action, using only a
	limited number of local atomic additions.  Its smaller coefficient traffic and
	higher arithmetic intensity make the element representation faster on the GPU\@.
	Table~\ref{tab:cpu-gpu-pointsource} therefore compares the fastest tested
	operator representation on each architecture: a stored row stencil on CPUs and
	an element-based matrix-free operator on GPUs.
	
	\begin{table}[H]
		\centering
		\caption{Architecture-optimized CPU--GPU comparison for the homogeneous
			point-source test.  The table reports the fastest tested implementation on each
			platform: stored row-stencil MatShell on CPUs and element-based matrix-free
			apply on GPUs for the $768^3$ problem.  The CPU run uses AMD EPYC 7452 nodes
			with two CPU cores per MPI rank, and the GPU run uses eight A100-40G GPUs.  The
			memory column is the total peak memory over all CPUs or GPUs.}
		\resizebox{\textwidth}{!}{%
			\begin{tabular}{cccccccccc}
				\toprule
				platform & hardware & operator representation & mesh & resources & PC calls
				& $\|A_h\bm U_h-\bm F_h\|_2/\|\bm F_h\|_2$ & setup & solve & peak memory \\
				\midrule
			GPU & NVIDIA A100-40G & element matrix-free & $768^3$ & 8 GPUs & 30
			& $8.24\times 10^{-4}$ & 0.221 s & 17.668 s & 296.8 GiB \\
			CPU & AMD EPYC 7452 & stored-stencil MatShell & $768^3$
			& 24 nodes, 768 MPI ranks $\times$ 2 threads & 30
			& $8.41\times 10^{-4}$ & 11.88 s & 727.86 s & 2101.61 GiB \\
				\bottomrule
			\end{tabular}
		}
		\label{tab:cpu-gpu-pointsource}
	\end{table}
	
	The CPU and GPU runs require the same number of preconditioner applications and
	reach essentially the same true residual, confirming that the two implementations
	realize the same algebraic algorithm.  The CPU run binds one MPI rank to two cores
	and uses 768 MPI ranks.  Relative to this baseline, the eight-GPU solve is
	$41.2$ times faster and uses $7.1$ times less peak memory.  These numbers compare
	the fastest tested realization on each architecture; they do not
	imply that either the element or row-stencil representation is universally
	preferable.
	
	\section{Large-Scale GPU Experiments}
	\label{sec:performance}
	
	We consider four three-dimensional Maxwell problems: a homogeneous reference
	medium, a smooth converging lens, a rapidly
	oscillatory periodic medium, and a three-dimensionally varying random isotropic
	medium.  For each medium, we increase the total mesh from $384^3$ on one NVIDIA
	A100-40G GPU to $1536^3$ on 64 GPUs while keeping the spatial resolution and PML
	thickness fixed.  The physical propagation distance therefore grows with the
	mesh, and the largest run contains approximately $10.89$ billion complex edge
	unknowns.  We record the preconditioner applications, matrix-free setup time,
	solve time, time per preconditioner application, and total peak GPU memory for
	every problem size.

	We then hold the mesh and coefficient field fixed and vary only the number of
	GPUs.  Strong scaling is measured on the $384^3$, $768^3$, and $1224^3$
	problems, whereas the volume-matched runs at $384^3/1$, $768^3/8$, and
	$1536^3/64$ form a weak-scaling sequence.  Together, these experiments cover
	smooth propagation, focusing, repeated scattering, and fully three-dimensional
	coefficient variation with the same discretization and preconditioner
	configuration.  The common experimental protocol is specified next.
	
	\subsection{Experimental setup}
	\label{sec:experimental-protocol}
	
	All performance results use one discretization and one solver configuration;
	no parameter is retuned for a particular medium, mesh, or GPU count.  Thus,
	differences in the preconditioner count directly reflect changes in propagation
	distance and material structure, rather than case-dependent solver choices.
	Section~5 independently justifies the selected outer tolerance and arithmetic
	precision: tightening the relative residual below $10^{-3}$ has negligible
	effect on the receiver profiles, while complex single precision preserves the
	preconditioner count and produces negligible change in the computed field.
	Every total mesh includes twelve PML cells on each side.
	
	All runs use the same matrix-free CUDA and MPI implementation with a two-dimensional
	$y$--$z$ pencil decomposition; only nearest-neighbor halo data and Krylov
	reductions are communicated.  To control placement variability, each scaling
	sequence uses prefix-nested allocations from a fixed rank-to-node ordering, and
	one run of each configuration is reported.
	
		\begin{table}[H]
		\centering
		\small
		\renewcommand{\arraystretch}{0.90}
		\caption{Fixed discretization and three-grid solver parameters used in the performance experiments.}
		\begin{tabular}{@{}p{0.34\linewidth}p{0.56\linewidth}@{}}
			\toprule
			item & setting \\
			\midrule
			finite element space & lowest-order N{\'e}delec elements on hexahedra \\
			quadrature & blended rule for both fine and coarse operators \\
			mesh convention & total mesh includes the PML cells \\
			resolution & minimum physical-domain $\mathrm{ppw}=8$, $n_{\rm pml}=12$ \\
			target operator & unshifted Maxwell matrix $A_h$ \\
			outer Krylov method & FGMRES(4) applied to $A_h\bm U_h=\bm F_h$ \\
			outer stopping test & $\|A_h\bm U_h-\bm F_h\|_2/\|\bm F_h\|_2 \le 10^{-3}$ \\
			shifted hierarchy & shifted Maxwell operators $B_{2h}$ and $B_{4h}$ with $\sigma=0.4$ \\
			fine-grid smoother & fixed pre- and post-GMRES(4)--Jacobi(2), $\omega_h^{\mathrm{J}}=0.55$ \\
			$2h$ coarse equation & four fixed inner FGMRES iterations on the unshifted $A_{2h}$ equation, restart length 10 \\
			$2h$ inner preconditioner & shifted V-cycle using GMRES(4)--Jacobi(2) smoothing on $B_{2h}$, $\omega_{2h}^{\mathrm{J}}=0.55$ \\
			$4h$ bottom correction & fixed GMRES(10)--Jacobi(2) applied to $B_{4h}$, $\omega_{4h}^{\mathrm{J}}=0.4$ \\
			precision & complex single precision \\
			\bottomrule
		\end{tabular}
		\label{tab:performance-protocol}
	\end{table}
	
	Here outer FGMRES(4) has restart length four, whereas
	GMRES($m$)--Jacobi(2) denotes $m$ fixed GMRES steps, each preconditioned by two
	damped Jacobi sweeps.  In the notation of Section~3, the fixed work counts are
	\[
		m_h=4,\quad \nu_h=2,\quad m_{\rm mid}=4,\quad
		m_{2h}=4,\quad \nu_{2h}=2,\quad
		m_{4h}=10,\quad \nu_{4h}=2.
	\]
	The pre- and post-smoothing calls on a given level use the same listed budget.
	Shifts appear only in the auxiliary hierarchy;
	the outer FGMRES iteration, stopping criterion, and reported residual always
	use the unshifted physical operator $A_h$.
	
	The homogeneous problem uses the $z$-oriented N{\'e}delec point load from the
	Green-tensor validation.  The heterogeneous problems use the $z$-oriented
	source of the Maxwell sweeping experiments \cite{Tsuji2012}, centered at
	$(0,0.5,0)$.  The mesh/resource pairs are $384^3/1$, $768^3/8$,
	$1224^3/32$, and $1536^3/64$, where the second number denotes the GPU count.
	These pairs keep the fine-grid volume per GPU fixed to within $1.2\%$.  We
	therefore report the solve-only weak-scaling efficiency based on the time per
	preconditioner application, equivalently the per-iteration cost of outer FGMRES,
	\[
		\eta_{\rm weak}(N,n_{\mathrm{GPU}})
		=
		\frac{(T_{\rm solve}/n_{\rm PC})_{384^3,1}}
		{(T_{\rm solve}/n_{\rm PC})_{N^3,n_{\mathrm{GPU}}}},
	\]
	where $N$ is the total mesh size in each coordinate direction,
	$n_{\mathrm{GPU}}$ is the GPU count, and $n_{\rm PC}$ is the number of
	preconditioner applications.  The
		matrix-free hierarchy requires neither sparse matrix assembly nor a direct
		factorization.  Setup builds the coefficient and diagonal data together with
		transfer and communication metadata, and remains below one second in every
		experiment.

	\begin{table}[t]
		\centering
		\footnotesize
		\setlength{\tabcolsep}{2.6pt}
		\renewcommand{\arraystretch}{0.92}
		\caption{Mesh and resource scaling for all four media on NVIDIA A100-40G
		GPUs.  Weak-scaling efficiency is computed from $T_{\rm solve}$/PC relative
		to the $384^3$ one-GPU case for each medium; memory is the total peak over
		all GPUs.}
		\begin{tabular}{@{}llccccccc@{}}
			\toprule
			medium & mesh & PC calls & GPUs & $T_{\rm setup}$ & $T_{\rm solve}$ &
			$T_{\rm solve}$/PC & $\eta_{\rm weak}$ & memory \\
			\midrule
			homogeneous & $384^3$  & 18 & 1  & 0.044 s & 9.582 s  & 0.532 s & 100.0\% & 37.1 GiB \\
			homogeneous & $768^3$  & 30 & 8  & 0.221 s & 17.668 s & 0.589 s & 90.4\%  & 296.8 GiB \\
			homogeneous & $1224^3$ & 44 & 32 & 0.335 s & 31.692 s & 0.720 s & 73.9\%  & 1204.6 GiB \\
			homogeneous & $1536^3$ & 55 & 64 & 0.402 s & 42.047 s & 0.764 s & 69.6\%  & 2381.4 GiB \\
			\hdashline\noalign{\vskip 2pt}
			lens & $384^3$  & 23 & 1  & 0.082 s & 12.429 s & 0.540 s & 100.0\% & 37.1 GiB \\
			lens & $768^3$  & 47 & 8  & 0.270 s & 28.009 s & 0.596 s & 90.7\%  & 296.8 GiB \\
			lens & $1224^3$ & 76 & 32 & 0.478 s & 52.598 s & 0.692 s & 78.1\%  & 1204.6 GiB \\
			lens & $1536^3$ & 98 & 64 & 0.492 s & 72.047 s & 0.735 s & 73.5\%  & 2381.4 GiB \\
			\hdashline\noalign{\vskip 2pt}
			periodic & $384^3$  & 21 & 1  & 0.149 s & 11.510 s & 0.548 s & 100.0\% & 37.1 GiB \\
			periodic & $768^3$  & 35 & 8  & 0.292 s & 21.450 s & 0.613 s & 89.4\%  & 296.8 GiB \\
			periodic & $1224^3$ & 56 & 32 & 0.486 s & 39.564 s & 0.706 s & 77.6\%  & 1204.6 GiB \\
			periodic & $1536^3$ & 70 & 64 & 0.501 s & 54.652 s & 0.781 s & 70.2\%  & 2381.4 GiB \\
			\hdashline\noalign{\vskip 2pt}
			random & $384^3$  & 23 & 1  & 0.412 s & 13.092 s & 0.569 s & 100.0\% & 37.1 GiB \\
			random & $768^3$  & 38 & 8  & 0.688 s & 23.774 s & 0.626 s & 91.0\%  & 296.8 GiB \\
			random & $1224^3$ & 57 & 32 & 0.745 s & 40.449 s & 0.710 s & 80.2\%  & 1204.6 GiB \\
			random & $1536^3$ & 71 & 64 & 0.750 s & 55.178 s & 0.777 s & 73.2\%  & 2381.4 GiB \\
			\bottomrule
		\end{tabular}
		\label{tab:a100-mesh-weak-scaling}
	\end{table}
		
	\subsection{Homogeneous medium}
	
		The homogeneous sequence gives the cleanest measure of algorithmic scaling with
		propagation distance.  Because the points per wavelength are fixed, increasing
		the mesh from $384^3$ to $1536^3$ increases the domain diameter measured in
		wavelengths by approximately a factor of four without changing the local
		discretization quality.

		The homogeneous rows of Table~\ref{tab:a100-mesh-weak-scaling} show controlled
		growth of the preconditioner count: a fourfold increase in propagation distance
		requires only about a threefold increase, from 18 to 55 applications.  Along the
		volume-matched resource sequence, the time per application rises from $0.532$ to
		$0.764$ seconds.  The $1536^3$ system is consequently solved in $42.0$ seconds
		on 64 GPUs, while construction of the matrix-free hierarchy takes $0.40$ seconds.
	\subsection{Converging lens}
	
	The converging lens, adapted from the Maxwell sweeping experiments
	\cite{Tsuji2012}, is a smooth but strongly varying medium that bends and focuses
	the propagating field.  It is the most demanding of the four examples because
	the coarse hierarchy must represent both a spatially varying local wavelength
	and strongly refracted propagation paths.  The material parameters are
	\[
	\epsilon_r(\bm{x})=\mu_r(\bm{x})
	=
	\left[
	\frac{3}{4}
	\left(1-\frac{1}{2}
	\exp\{-8(x^2+y^2+z^2)\}\right)
	\right]^{-1}.
	\]
	The source location and orientation follow \cite{Tsuji2012}.

	Even for this most demanding medium, the $1536^3$ problem remains below 100
	preconditioner applications and is solved in $72.0$ seconds.  Its $0.735$ seconds
	per application is close to the homogeneous value of $0.764$
	seconds: refraction increases the number of coarse-grid corrections, but not the
	throughput of the matrix-free kernels.  The left panel of
	Figure~\ref{fig:heterogeneous768-wavefields} shows the
		physical complexity of this test.  The nearly
		spherical wavefronts emitted by the off-center source are refracted as they
		enter the lens, compressed near the symmetry axis, and reorganized into a
		pronounced downstream focusing pattern.  Although the strongest coefficient
		variation is concentrated near the lens, its effect persists throughout the
		downstream wavefield.

	\subsection{Periodic medium}
	
	The periodic medium from \cite{Tsuji2012} replaces the localized smooth lens by
	rapidly oscillating coefficients throughout the domain.  It subjects the coarse
	hierarchy to repeated scattering, interference, and continuously varying local
	element operators.
	Define
	\[
	m_{\rm per}(x,y)
	=
	1
	+\frac{1}{4}
	\cos\left(20\left(\frac{x}{\sqrt{2}}+\frac{y}{\sqrt{2}}\right)\right)
	+\frac{1}{4}
	\cos\left(20\left(\frac{x}{\sqrt{2}}-\frac{y}{\sqrt{2}}\right)\right),
	\]
	and set
	\[
	\epsilon_r(\bm{x})=\mu_r(\bm{x})=\sqrt{m_{\rm per}(x,y)}.
	\]
	The medium is invariant in the $z$ direction, and the source is the same as in
	the lens problem.

	Repeated scattering produces only a moderate increase in numerical difficulty.
	The $1536^3$ system converges in 70 applications, compared with 55 for the
	homogeneous medium, and is solved in $54.7$ seconds.  At every mesh size, its time
	per application differs by at most $4.1\%$ from the homogeneous value.  The center
	panel of Figure~\ref{fig:heterogeneous768-wavefields}
		shows that the outgoing wavefronts are no longer nearly spherical.  The periodic
		coefficient produces oblique bands of enhanced amplitude and repeated phase
		modulation across the source-crossing planes, revealing the preferred scattering
		directions created by the two rotated coefficient modes.

	\begin{figure}[!b]
		\centering
		\includegraphics[width=0.32\textwidth,trim=10bp 18bp 0 30bp,clip]{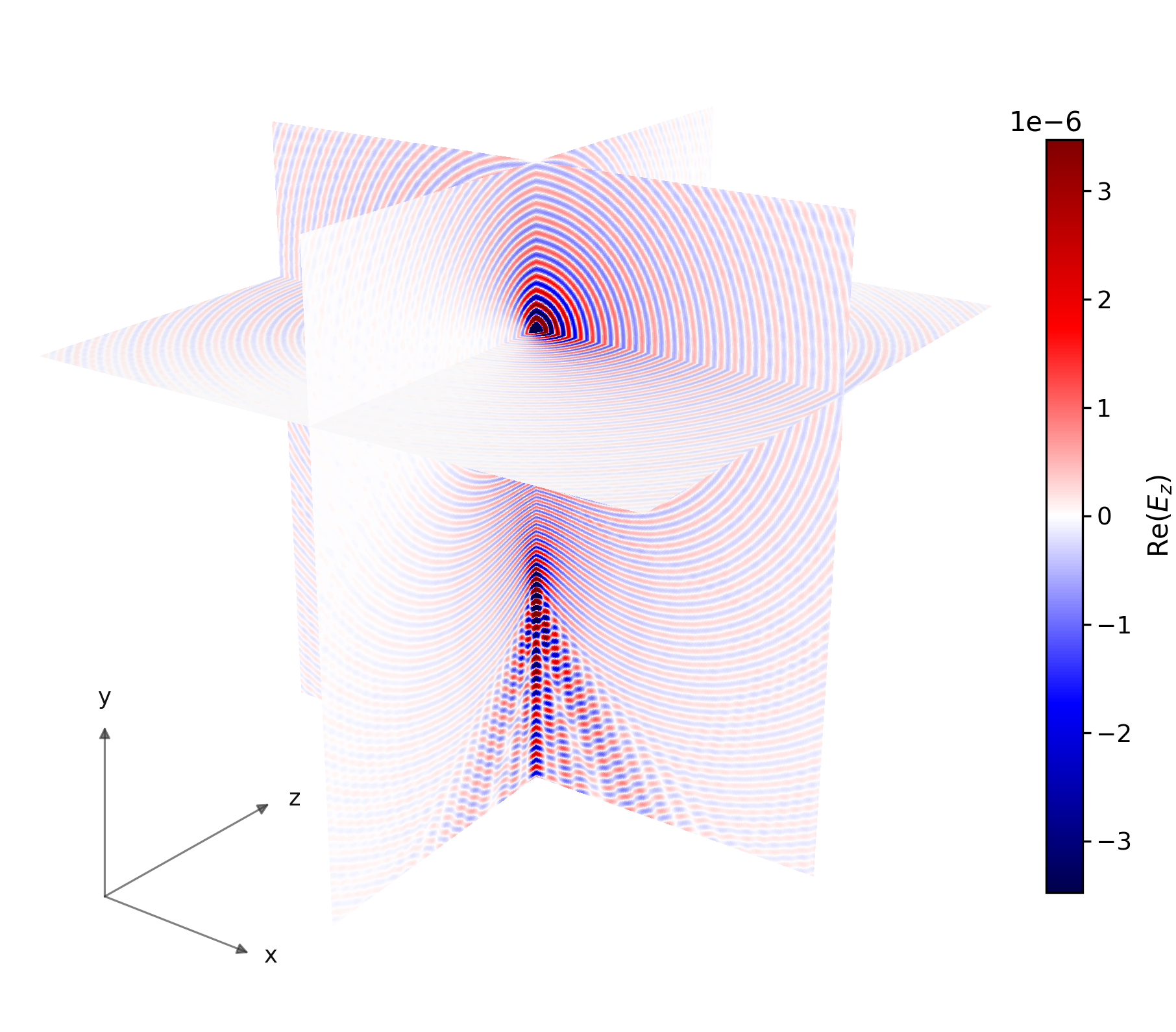}\hfill
		\includegraphics[width=0.32\textwidth,trim=10bp 18bp 0 30bp,clip]{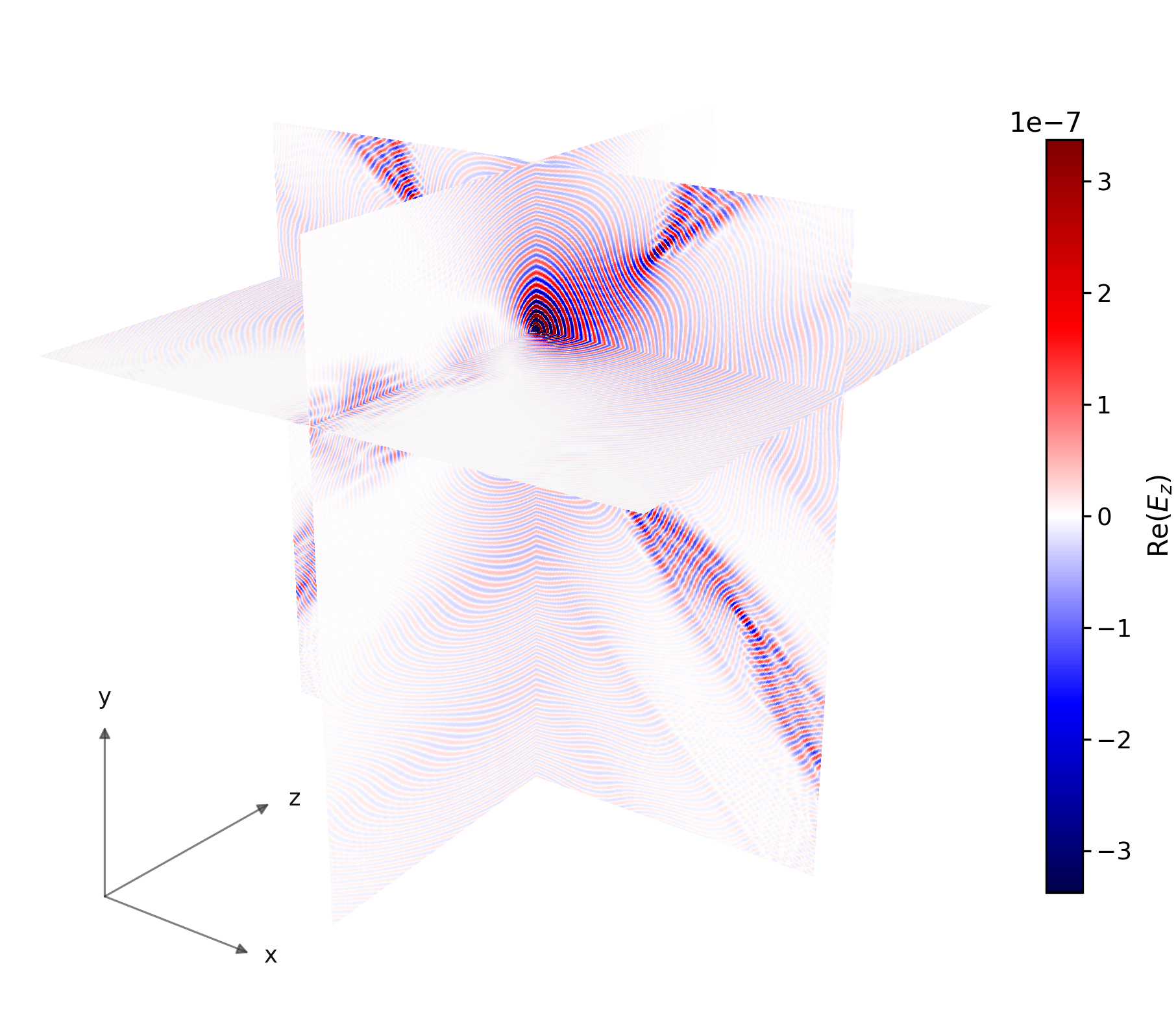}\hfill
		\includegraphics[width=0.32\textwidth,trim=10bp 18bp 0 30bp,clip]{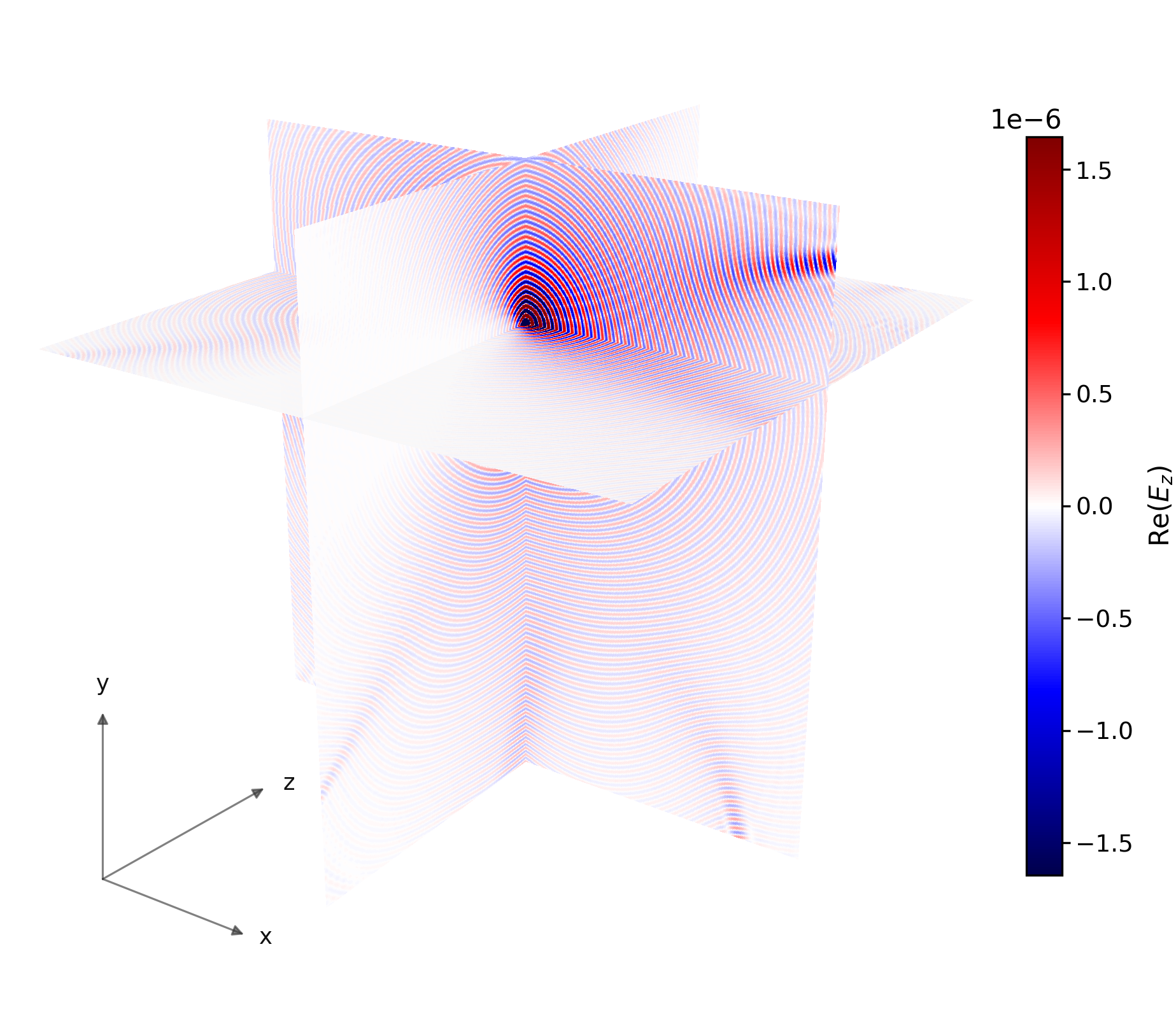}
		\caption{Source-crossing views of $\operatorname{Re}(E_z)$ on $768^3$ total meshes:
		converging lens (left), periodic medium (center), and random isotropic medium
		(right).  Each panel uses one common color scale for its three planes; PML cells
		are omitted.}
		\label{fig:heterogeneous768-wavefields}
	\end{figure}

	\subsection{Random isotropic medium}
	
	The random medium removes both the radial structure of the lens and
	the translational structure of the periodic medium.  It provides the strongest
	check that the observed performance does not depend on symmetry, separability,
	or a repeating coefficient pattern.  Its deterministic three-dimensional
	coefficient field is defined by
	\[
	\begin{aligned}
		\rho(\bm{x})={}&
		0.71\cos(2\pi(0.50x+y-0.25z)+0.37)\\
		&+0.53\sin(2\pi(1.25x-0.50y+0.75z)+1.11)\\
		&+0.47\cos(2\pi(-0.75x+1.50y+0.50z)+2.23)\\
		&+0.41\sin(2\pi(1.50x+0.25y-z)+3.17)\\
		&+0.37\cos(2\pi(0.25x-1.25y+1.50z)+4.03)\\
		&+0.31\sin(2\pi(-1.50x+0.75y+1.25z)+5.19)\\
		&+0.29\cos(2\pi(x+y+z)+0.83)\\
		&+0.23\sin(2\pi(-1.25x-0.75y+0.50z)+2.71).
	\end{aligned}
	\]
	We set
	\[
	\chi(\bm{x}) = 0.5 + 0.45\, \rho(\bm{x})/3.32,
	\qquad
	\epsilon_r(\bm{x})=\mu_r(\bm{x})
	=\frac{4/3}{1-\chi(\bm{x})/2}.
	\]
	The source is unchanged from the preceding heterogeneous examples.

	The fully three-dimensional coefficient field requires 23--71 preconditioner
	applications over the four meshes and is handled almost as efficiently as the
	periodic medium.  The $1536^3$ problem is solved in $55.2$ seconds.  This close
	agreement shows that the method does not rely on radial symmetry, separability,
	or periodic repetition of the material distribution.  In the right panel of
	Figure~\ref{fig:heterogeneous768-wavefields}, the
		globally outgoing field is retained, but the wavefront spacing, curvature, and
		amplitude vary from one region to another.  These localized distortions on all
		three planes are the visible signature of a genuinely three-dimensional
		coefficient field rather than a two-dimensional extrusion.

	Across all four media, the weak-scaling efficiency is $89.4$--$91.0\%$ on
	eight GPUs, $73.9$--$80.2\%$ on 32 GPUs, and $69.6$--$73.5\%$ on 64 GPUs.
	
	\FloatBarrier
	\subsection{Strong scaling}

	Strong scaling is reproducible across all four coefficient fields.
	Only the GPU count is varied; the
	discretization, physical data, stopping criterion, and all three-grid parameters
	remain fixed.  The preconditioner count is unchanged in most comparisons and
	differs by at most one application when the true residual crosses the fixed
	stopping threshold.  The timing ratios therefore measure the parallel execution
	of nearly identical numerical work rather than a retuned solver.
	Parallel efficiency is computed from solve time alone.  With
	$n_{\mathrm{GPU},0}$ denoting the baseline allocation, we use
	\[
		\eta_{\rm strong}(n_{\mathrm{GPU}})
		=
		\frac{n_{\mathrm{GPU},0}T_{\rm solve}(n_{\mathrm{GPU},0})}
		{n_{\mathrm{GPU}}T_{\rm solve}(n_{\mathrm{GPU}})}.
	\]
	The baselines are one GPU for the $384^3$ mesh, eight GPUs for $768^3$, and
	32 GPUs for $1224^3$.

	\begin{table}[H]
		\centering
		\caption{Strong scaling on a fixed $384^3$ total mesh using A100-40G GPUs.}
		\resizebox{\textwidth}{!}{%
		\begin{tabular}{cccccccc}
			\toprule
			medium & mesh & PC calls & GPUs & $T_{\rm setup}$ & $T_{\rm solve}$ & $\eta_{\rm strong}$ & memory \\
			\midrule
			homogeneous & $384^3$ & 18 & 1 & 0.044 s & 9.582 s & 100.0\% & 37.1 GiB \\
			homogeneous & $384^3$ & 18 & 2 & 0.044 s & 5.114 s & 93.7\% & 37.4 GiB \\
			homogeneous & $384^3$ & 18 & 4 & 0.051 s & 3.253 s & 73.6\% & 37.5 GiB \\
			\hdashline\noalign{\vskip 2pt}
			lens        & $384^3$ & 23 & 1 & 0.082 s & 12.429 s & 100.0\% & 37.1 GiB \\
			lens        & $384^3$ & 23 & 2 & 0.085 s & 6.654 s & 93.4\% & 37.4 GiB \\
			lens        & $384^3$ & 23 & 4 & 0.090 s & 4.146 s & 74.9\% & 37.5 GiB \\
			\hdashline\noalign{\vskip 2pt}
			periodic    & $384^3$ & 21 & 1 & 0.149 s & 11.510 s & 100.0\% & 37.1 GiB \\
			periodic    & $384^3$ & 21 & 2 & 0.147 s & 6.137 s & 93.8\% & 37.4 GiB \\
			periodic    & $384^3$ & 21 & 4 & 0.157 s & 3.874 s & 74.3\% & 37.5 GiB \\
			\hdashline\noalign{\vskip 2pt}
			random      & $384^3$ & 23 & 1 & 0.412 s & 13.092 s & 100.0\% & 37.1 GiB \\
			random      & $384^3$ & 23 & 2 & 0.411 s & 7.024 s & 93.2\% & 37.4 GiB \\
			random      & $384^3$ & 23 & 4 & 0.417 s & 4.546 s & 72.0\% & 37.5 GiB \\
			\bottomrule
		\end{tabular}%
		}
		\label{tab:a100-strong-scaling-384}
	\end{table}

	\begin{table}[H]
		\centering
		\caption{Strong scaling on a fixed $768^3$ total mesh using A100-40G GPUs.}
		\resizebox{\textwidth}{!}{%
		\begin{tabular}{cccccccc}
			\toprule
			medium & mesh & PC calls & GPUs & $T_{\rm setup}$ & $T_{\rm solve}$ & $\eta_{\rm strong}$ & memory \\
			\midrule
			homogeneous & $768^3$ & 30 & 8  & 0.221 s & 17.668 s & 100.0\% & 296.8 GiB \\
			homogeneous & $768^3$ & 30 & 16 & 0.237 s & 11.115 s & 79.5\% & 299.0 GiB \\
			homogeneous & $768^3$ & 30 & 32 & 0.277 s & 7.715 s & 57.3\% & 300.2 GiB \\
			\hdashline\noalign{\vskip 2pt}
			lens        & $768^3$ & 47 & 8  & 0.270 s & 28.009 s & 100.0\% & 296.8 GiB \\
			lens        & $768^3$ & 47 & 16 & 0.293 s & 17.485 s & 80.1\% & 299.0 GiB \\
			lens        & $768^3$ & 47 & 32 & 0.317 s & 11.446 s & 61.2\% & 300.2 GiB \\
			\hdashline\noalign{\vskip 2pt}
			periodic    & $768^3$ & 35 & 8  & 0.292 s & 21.450 s & 100.0\% & 296.8 GiB \\
			periodic    & $768^3$ & 35 & 16 & 0.414 s & 13.299 s & 80.6\% & 299.0 GiB \\
			periodic    & $768^3$ & 35 & 32 & 0.377 s & 8.715 s & 61.5\% & 300.2 GiB \\
			\hdashline\noalign{\vskip 2pt}
			random      & $768^3$ & 38 & 8  & 0.688 s & 23.774 s & 100.0\% & 296.8 GiB \\
			random      & $768^3$ & 38 & 16 & 0.747 s & 14.390 s & 82.6\% & 299.0 GiB \\
			random      & $768^3$ & 38 & 32 & 0.609 s & 9.845 s & 60.4\% & 300.2 GiB \\
			\bottomrule
		\end{tabular}%
		}
		\label{tab:a100-strong-scaling-768}
	\end{table}

	\begin{table}[H]
		\centering
		\caption{Strong scaling on a fixed $1224^3$ total mesh using A100-40G GPUs.}
		\resizebox{\textwidth}{!}{%
		\begin{tabular}{cccccccc}
			\toprule
			medium & mesh & PC calls & GPUs & $T_{\rm setup}$ & $T_{\rm solve}$ & $\eta_{\rm strong}$ & memory \\
			\midrule
			homogeneous & $1224^3$ & 44 & 32 & 0.335 s & 31.692 s & 100.0\% & 1204.6 GiB \\
			homogeneous & $1224^3$ & 44 & 64 & 0.332 s & 19.707 s & 80.4\% & 1211.8 GiB \\
			\hdashline\noalign{\vskip 2pt}
			lens        & $1224^3$ & 76 & 32 & 0.478 s & 52.598 s & 100.0\% & 1204.6 GiB \\
			lens        & $1224^3$ & 76 & 64 & 0.373 s & 32.948 s & 79.8\% & 1211.8 GiB \\
			\hdashline\noalign{\vskip 2pt}
			periodic    & $1224^3$ & 56 & 32 & 0.486 s & 39.564 s & 100.0\% & 1204.6 GiB \\
			periodic    & $1224^3$ & 56 & 64 & 0.447 s & 25.346 s & 78.0\% & 1211.8 GiB \\
			\hdashline\noalign{\vskip 2pt}
			random      & $1224^3$ & 57 & 32 & 0.745 s & 40.449 s & 100.0\% & 1204.6 GiB \\
			random      & $1224^3$ & 57 & 64 & 0.693 s & 25.815 s & 78.3\% & 1211.8 GiB \\
			\bottomrule
		\end{tabular}%
		}
		\label{tab:a100-strong-scaling-1224}
	\end{table}
	\FloatBarrier

	Strong scaling is consistent across all four media.  The $384^3$ problems retain
	$93.2\%$--$93.8\%$ efficiency on two GPUs and $72.0\%$--$74.9\%$ on four;
	across nodes, the $768^3$ and $1224^3$ problems retain $79.5\%$--$82.6\%$ and
	$78.0\%$--$80.4\%$, respectively, when the baseline allocation is doubled.
	The agreement among the media shows that communication, rather than coefficient
	complexity, governs the parallel behavior.

	With fixed fine-grid volume per GPU, the $768^3$ and $1536^3$ runs achieve
	$89.4\%$--$91.0\%$ and $69.6\%$--$73.5\%$ efficiency relative to $384^3$ on one
	GPU.  The remaining loss at 64 GPUs comes from halo exchange and global Krylov
	reductions.  Because the cluster neither guarantees topology-aware placement nor
	provides direct GPU halo transfers, these cross-node measurements are conservative.

	\FloatBarrier
	\section{Conclusions}
	
	We developed a fully matrix-free three-grid preconditioner that confines complex
	shifting to an auxiliary inner cycle while preserving the physical Maxwell
	operators.  Compatible edge transfers, blended quadrature, and fixed-work
	smoothers eliminate assembled matrices and a direct coarse solve, while the Fourier
	analysis and Green-tensor experiment support the parameter and accuracy choices.
	One configuration handles all four media and scales to $10.89$ billion complex
	unknowns, solving the largest systems in $42.0$--$72.0$ seconds on only 64
	A100-40G GPUs.  This combination of scale, speed, and modest resources is the
	central computational result.
	
	\bibliographystyle{siamplain}
	\bibliography{ref}
	
\end{document}